%% file: Testing_Final.tex
\documentclass[11pt]{article}
\usepackage{fullpage}

\usepackage{amsmath, amsthm, amssymb, mathtools, thmtools, mathdots,mathrsfs}
\usepackage{psfrag}
\usepackage{enumerate} 
\usepackage{setspace}
\usepackage{float} 
\usepackage{color}
\usepackage{array}
\usepackage{mathrsfs}
\usepackage{hyperref}
\usepackage{pgf,tikz}
\usepackage{mathrsfs}

\usepackage{pythontex}

\usetikzlibrary{arrows}

\definecolor{ffqqqq}{rgb}{1.,0.,0.}
\definecolor{xfqqff}{rgb}{0.5, 0.0, 1.0}

\input{AMS_commands}

\input{STAT_macros}

\definecolor{mypink}{rgb}{0.858, 0.188, 0.478}

\newcommand{\Prob}{\mathbb{P}}

\newcommand{\epsopt}{\epsilon_{\text{{\tiny OPT}}}}

\newcommand{\SPEC}{\ensuremath{G}}

\newcommand{\UNIERR}{\ensuremath{\operatorname{Err}}}

\newcommand{\Hyp}{\ensuremath{\mathcal{H}}}

\newcommand{\KposSet}{\mathcal{S}}

\newcommand{\trunc}[1]{\Pi_{#1}}
\newcommand{\testTr}{\Psi_{\text{{\tiny LPT}}}}

\newcommand{\ellip}{\mathcal{E}}
\newcommand{\tmpdim}{t}
\newcommand{\upk}{k_u}
\newcommand{\lwk}{k_\ell}
\newcommand{\kb}{k_B}

\newcommand{\epslpt}{\epsilon_{\text{{\tiny LPT}}}}

\newcommand{\at}{\theta^*}
\newcommand{\thetadag}{\theta^{\dagger}}

\newcommand{\epscritu}{\ensuremath{\epsilon_u}}

\newcommand{\newepscritu}{t^*_u}
\newcommand{\newepscritl}{t^*_{\ell}}

\newcommand{\epscritl}{\ensuremath{\epsilon_\ell}}
\newcommand{\epscritb}{\ensuremath{\epsilon_B}}

\newcommand{\LPset}[1]{\ensuremath{\mathcal{P}_{#1}}}

\newcommand{\plaindel}{\ensuremath{\delta}}
\newcommand{\ccon}{\ensuremath{c}}
\newcommand{\upkstar}{\ensuremath{k^*_u}}
\newcommand{\Ellipse}{\ellip}

\newcommand{\kwidth}{\ensuremath{\omega}}
\newcommand{\DeltaTil}{\tilde{\Delta}}

\newcommand{\enorm}[1]{\ensuremath{\|#1\|_{\Ellipse}}}

\DeclarePairedDelimiter{\braces}{\{}{\}}

\DeclarePairedDelimiter{\norm}{\|}{\|}

\DeclarePairedDelimiter{\inner}{\langle}{\rangle}

\newcommand{\Mat}{H}
\newcommand{\NewMat}{M}

\newcommand{\epsu}{\widehat{\epsilon}_u}
\newcommand{\epsl}{\widehat{\epsilon}_\ell}
\newcommand{\kestu}{\hat{k}_u}
\newcommand{\kestl}{\hat{k}_\ell}

\newcommand{\Rfun}{\ensuremath{\Phi}}
\newcommand{\InvRfun}{\Rfun^{-1}}

\newcommand{\newupk}{m_u}
\newcommand{\newlwk}{m_\ell}

\newcommand{\newkb}{\widetilde{k}_B}

\newcommand{\uhat}{\ensuremath{\widehat{u}}}
\newcommand{\ustar}{\ensuremath{u^*}}

\newcommand{\mind}{m}
\newcommand{\Mdex}{\mathcal{M}}
\newcommand{\AMat}{A}

\newcommand{\Nat}{\ensuremath{\mathbb{N}}}
\newcommand{\Xspace}{\ensuremath{\mathcal{X}}}

\newcommand{\Ker}{\ensuremath{\mathcal{K}}}

\newcommand{\numvec}{\ensuremath{N}}

\newcommand{\mindnots}{\ensuremath{\mind \backslash s}}

\begin{document}


\begin{center}

{\bf \LARGE{The local geometry of testing in ellipses: \\Tight control
    via localized Kolmogorov widths}}

\vspace*{.2in}

{\large{
\begin{tabular}{ccccc}
Yuting Wei$^\dagger$ &&  Martin J. Wainwright$^{\dagger, \star}$
\end{tabular}
}}

\vspace*{.2in} 
\today

 \begin{tabular}{c}
 Department of Statistics$^\dagger$, and \\ Department of Electrical
 Engineering and Computer Sciences$^\star$ \\ UC Berkeley, Berkeley,
 CA 94720
 \end{tabular}
 \vspace*{.2in}

\begin{abstract}
We study the local geometry of testing a mean vector within a
high-dimensional ellipse against a compound alternative. Given samples
of a Gaussian random vector, the goal is to distinguish whether the
mean is equal to a known vector within an ellipse, or equal to some
other unknown vector in the ellipse.  Such ellipse testing problems
lie at the heart of several applications, including non-parametric
goodness-of-fit testing, signal detection in cognitive radio, and
regression function testing in reproducing kernel Hilbert spaces.
While past work on such problems has focused on the
difficulty in a global sense, we study difficulty in a way that is
localized to each vector within the ellipse.  Our main result is to
give sharp upper and lower bounds on the localized minimax testing
radius in terms of an explicit formula involving the Kolmogorov width
of the ellipse intersected with a Euclidean ball.  When applied to
particular examples, our general theorems yield interesting rates that
were not known before: as a particular case, for testing in Sobolev
ellipses of smoothness $\alpha$, we demonstrate rates that vary from
$(\sigma^2)^{\frac{4 \alpha}{4 \alpha + 1}}$, corresponding to the
classical global rate, to the faster rate $(\sigma^2)^{\frac{8
    \alpha}{8 \alpha + 1}}$, achievable for vectors at favorable
locations within the ellipse.  We also show that the optimal test for
this problem is achieved by a linear projection test that is based on
an explicit lower-dimensional projection of the observation vector.
\end{abstract}
\end{center}


\vspace*{1cm}

\section{Introduction} 
\label{sec:Intr}

Testing and estimation are two fundamental classes of problems in
statistics~\cite{lehmann2006theory,lehmann2006testing}. In a classical
decision-theoretic framework, different methods are compared via their
risk functions, and under the minimax formalism, methods are compared
by their worst-case behavior over the entire parameter space.  Such
global minimax risk calculation are reflective of typical behavior
when the risk function is close to constant over the entire parameter
space.  On the other hand, when the risk function varies
substantially, then it is possible that a global minimax calculation
will be unduly conservative.  In such settings, one is motivated to
study the notion of an adaptive or localized minimax risk (e.g., see
references~\cite{donoho1994ideal,bickel1993efficient} for early work
in this vein).

Recent years have witnessed a rapidly evolving line of work on
studying notions of local minimax risk for both estimation and
testing.  In the context of shape-constrained regression, Meyer and
Woodruffe~\cite{meyer2000degrees} introduced a notion of degrees of
freedom that adapts to the geometry of the function being estimated.
Focusing on the problem of isotonic regression, Chatterjee et
al.~\cite{chatterjee2015risk} proved a range of convergence rates,
depending on the structure of the true regression function.  For the
task of estimating a convex function at a point, Cai and
Low~\cite{cai2015framework} proposed a local minimax criterion to
evaluate the performance of different estimators.  This criterion was
adapted to establish a form of instance-dependent optimality for
convex set estimation by Cai et al.~\cite{cai2015adaptive}.  In the
context of hypothesis testing, Valiant and
Valiant~\cite{valiant2014automatic} studied a class of compound
testing problems for discrete distributions, and characterized the
local minimax testing radius in the TV norm.  Balakrishnan and
Wasserman~\cite{balakrishnan2017hypothesis} studied an analogous
problem for testing density functions, and also characterized the
local testing radius in TV norm.

In this paper, we study the local geometry of testing a mean vector
inside an ellipse against a compound alternative. More precisely,
consider an ellipse of the form \mbox{$\Ellipse = \{ \theta \in
  \real^\usedim \mid \sum_{j=1}^\usedim \theta_j^2/\mu_j \leq 1 \}$,}
where $\{\mu_j\}_{j=1}^\usedim$ is a non-negative sequence that
defines the aspect ratios of the ellipse.  Although we focus on
finite-dimensional ellipses, all of our theory is non-asymptotic and
explicit, so that we can obtain results for infinite-dimensional
ellipses by taking suitable limits.  Given an observation of a
Gaussian random vector $y \in \real^\usedim$, our goal is to test
whether its mean is equal to some known vector $\at \in \Ellipse$, or
equal to some unknown vector $\theta \in \Ellipse$ that is suitably
separated from $\at$ in Euclidean norm.  As we discuss in more detail
below, such ellipse testing problems lie at the heart of a number of
applications, including non-parametric goodness-of-fit testing, signal
detection in cognitive radio, and testing regression functions in
kernel spaces.  

Study of the ellipse testing problems date back to the seminal work of
Ingster~\cite{ingster93a,ingster93b,ingster93c} and
Ermakov~\cite{ermakov1991minimax}, who focused their attention on the
special case when the null is zero ($\thetastar = 0$), and the ellipse
is induced by a Sobolev function class.  For this particular class of
testing problems, they provided a sharp characterization of the
minimax testing radius. A more general question is whether it is
possible to provide a geometric characterization of the minimax
testing radius, both as a function of the ellipse parameters
$\{\mu_j\}_{j=1}^\usedim$ as well as of the location of null vector
$\thetastar$ within the ellipse.  The main contribution of this paper
is to answer this question in the affirmative.  In particular, we show
how for any ellipse---including the Sobolev ellipses as special
cases---the localized minimax testing radius is characterized by a
formula that involves the Kolmogorov width of a particular set.  The
Kolmogorov width is a classical notion from approximation
theory~\cite{pinkus2012n}, which measures the ``size'' of a set in
terms of how well it can be approximated by a lower-dimensional linear 
subspace of
fixed dimension.  Our formula involves the Kolmogorov width of the
ellipse intersected with an Euclidean ball around the null vector
$\thetastar$, which leads to an equation involving both $\Ellipse$ and
$\thetastar$ that can be solved to determine the localized minimax
risk.  We show that the zero case ($\thetastar = 0$) is the most
difficult, in that the localized minimax radius is largest at 
this instance.  Conversely, we exhibit other more favorable locations
within ellipses for which the local minimax testing radius can be
substantially smaller.


\subsection{Some motivating examples}
\label{sec:motivatingExamples}

Before proceeding, let us consider some concrete examples so as to
motivate our study.


\begin{exas}[{\bf{Non-parametric goodness-of-fit testing}}]

Consider a set of samples $\{x_i\}_{i=1}^\numobs$ drawn from an
unknown distribution supported on a compact set $\Xspace$. Assuming
that the unknown distribution has a density, the goodness-of-fit
problem is to test whether the samples have been drawn from some fixed
density $p^*$, in which case the null hypothesis holds, or according
to some other density $p \neq p^*$, referred to as the alternative
hypothesis. In the non-parametric version of this problem, the
alternative density $p$ is allowed to vary over a broad function
class; one example might be the class of all twice continuously
differentiable densities with second derivative $p''$ belonging to the
unit ball in $L^2(\Xspace)$.  There is a very broad literature on this
topic, with the book of Ingster and
Suslina~\cite{ingster2012nonparametric} covering many different
classes of alternative densities.

One way to approach non-parametric goodness-of-fit is via orthogonal
series expansions.  In particular, if we let $\{\phi_j\}_{j=1}^\infty$
be an orthonormal basis for $L^2(\Xspace)$, then the null density
$p^*$ and alternative $p$ can be described in terms of their basis
expansion coefficients
\begin{align*}
\theta^*_j = \int_{\Xspace} \phi_j(x) p^*(x) dx \; = \;
\Exs_{p^*}[\phi_j(X)], \quad \mbox{and} \quad \theta_j =
\int_{\Xspace} \phi_j(x) p(x) dx \; = \; \Exs_{p}[\phi_j(X)],
\end{align*}
respectively, for $j = 1, 2, \ldots$. Note that the sample averages
$y_j \defn \frac{1}{\numobs} \sum_{i=1}^\numobs \phi_j(x_i)$ are
unbiased estimates of coefficients of the true underlying density, so
the testing problem can be written as
\begin{align*}
\Hyp_0: \quad y = \theta^* + w \quad \mbox{versus} \quad 
\Hyp_1: \quad y =
\theta + w, \quad \mbox{for some $\theta \neq \theta^*$}
\end{align*}
where $w = \{w_j\}_{j=1}^\infty$ is a sequence of noise variables.  A
typical smoothness constraint on the alternative---for instance, the
condition $\int_{\Xspace} (p''(x))^2 dx \leq 1$---amounts to requiring
that, in a suitably chosen sinusoidal basis, the vector $\theta$
belongs an ellipse with parameters of the form $\mu_j = c j^{-4}$.
\end{exas}


\begin{exas}[{\bf{Detection of unknown signals in noise}}]

In cognitive radio and other wireless applications
(e.g.,~\cite{atapattu2014energy,digham2007energy,shayegh2014signal}),
it is frequently of interest to test for the presence of unknown
signals that might potentially interfere with transmission.  More
concretely, given an observation period $[0,T]$, one observes a
continuous-time waveform $\{ y(t) \mid t \in [0,T] \}$, and the goal is to
determine whether it was generated by pure noise, or a by combination
of some band-limited signal $\theta(t)$ plus noise.  See
Figure~\ref{FigCognitive} for an illustration.

\begin{figure}[h]
\begin{center}
  \widgraph{.8\textwidth}{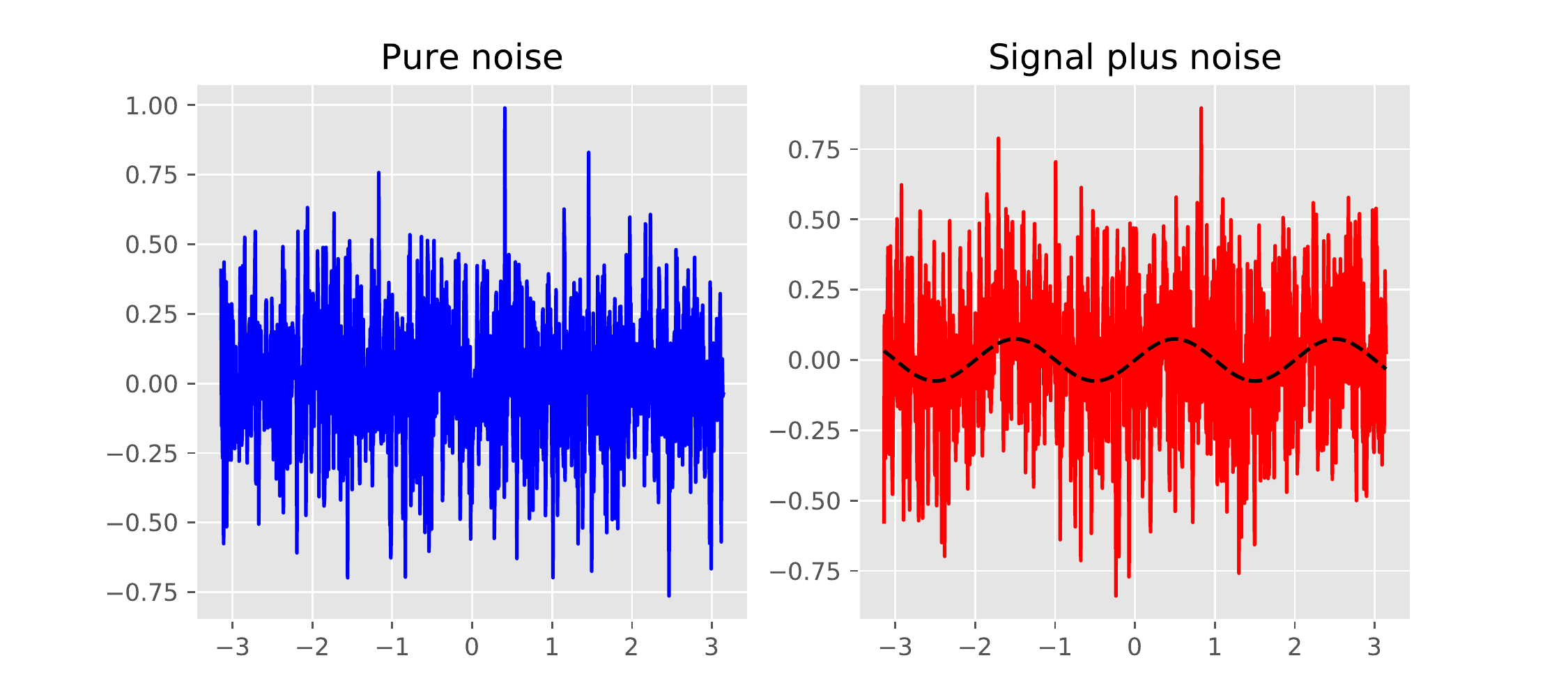}
  \caption{Illustration of the unknown signal detection problem in
    cognitive radio.  The left panel shows a continuous-time wave form
    that is pure noise, whereas the right panel shows the
    superposition of an unknown band-limited signal with a noise
    waveform.  The testing problem is to distinguish the null
    hypothesis (pure noise) from the compound alternative (some
    band-limited signal has been added).}
  \label{FigCognitive}
\end{center}
\end{figure}
After applying the Fourier transform, the waveform and signal can be
represented by sequences $y \in \ell^2(\Nat)$ and $\theta \in
\ell^2(\Nat)$, respectively.  Assuming that the Fourier frequencies
are ordered from smallest to largest, a hard band-width constraint on
the signal corresponds to the restriction that $\theta_j = 0$ for all
$j > B$, where the integer $B$ corresponds to the bandlimit cut-off.
Note that this bandwidth constraint defines a (degenerate) ellipse in
$\ell^2(\Nat)$.  More generally, a soft bandwidth constraint can be
imposed by requiring that the weighted sum $\sum_{j=1}^\infty
\frac{\theta_j^2}{\mu_j}$ for some summable sequence of positive
coefficients $\{\mu_j\}_{j=1}^\infty$.  As before, this constraint
means that the unknown signal vector $\theta$ must belong to an
ellipse $\Ellipse$.

Given this set-up, the problem of detecting the presence of an unknown
signal versus noise corresponds to testing the null that $y \sim
N(0, \sigma^2 I_d)$ versus the alternative $y \sim N(\theta,
\sigma^2 I_d)$ for some $\theta \in \Ellipse$.  More generally, one
might be interested in testing between the null hypothesis $y \sim
N(\thetastar, \sigma^2 I_\usedim)$, for some fixed signal vector
$\thetastar \in \Ellipse$, versus the same alternative.  We refer the
reader to the
papers~\cite{atapattu2014energy,digham2007energy,shayegh2014signal}
for further details on this application.

\end{exas}

\begin{exas}[{\bf{Regression function testing in kernel spaces}}]
\label{ExaRegression}
In the problem of fixed design regression, we observe pairs $(x_i,
z_i) \in \Xspace \times \real$ of the form
\begin{align}
  \label{EqnPointwise}
  z_i & = f(x_i) + \xi_i \quad \mbox{for $i = 1, \ldots, \numobs$}.
\end{align}
Here $f: \Xspace \rightarrow \real$ is an unknown function, and
$\{\xi_i\}_{i=1}^\numobs$ are a sequence of zero-mean noise variables
with variance $\sigma^2$.  Based on the observation vector $z \in
\real^\numobs$, we might be interested in testing various hypotheses
related to the unknown function $f$. For instance, is $f$ a non-zero
function?  More generally, does $f$ deviate significantly from
some fixed class of parametric functions?

These questions are meaningful only if we impose some structures on the
underlying function space, and given the
observations~\eqref{EqnPointwise} based on pointwise evaluation, one
of the most natural is based on reproducing kernel Hilbert spaces
(e.g.,~\cite{Wahba,Gu02}).  These function spaces are defined by a
symmetric positive semidefinite kernel function $\Ker: \Xspace
\times \Xspace \rightarrow \real$; as some simple examples, when
$\Xspace = [0,1]$, the kernel function $\Ker(x, x') = 1 + \min \{x,
x'\}$ defines a class of Lipschitz functions, whereas the Gaussian
kernel $\Ker(x, x') = \exp( -\frac{1}{2 t} (x - x')^2)$ with bandwidth
$t > 0$ defines a class of infinitely differentiable functions.

For any choice of kernel, the pointwise observation
model~\eqref{EqnPointwise} can be transformed to a sequence model
form.  In particular, define the kernel matrix $K \in \real^{\numobs
  \times \numobs}$ with entries $K_{ij} = \Ker(x_i, x_j)/\numobs$, and
let $\mu_1 \geq \mu_2 \geq \cdots \geq \mu_\numobs \geq 0$ represent
its ordered eigenvalues.  In terms of this notation, the pointwise
model with any function $f$ in the unit norm Hilbert ball can be
transformed to the model $y = \theta + w$, where the vector $\theta
\in \real^\numobs$ satisfies the ellipse constraint
$\sum_{j=1}^\numobs \frac{\theta_j^2}{\mu_j} \leq 1$, and each entry
of the noise vector is zero-mean with variance $\frac{\sigma^2}{n}$.
We have thus arrived at another instance of our general ellipse
problem.
\end{exas}


\subsection{Problem formulation}
\label{sec:Setup}

Having understood the range of motivations for our problem, 
let us set up the problem more precisely.  Given a sequence of
positive scalars $\mu_1\geq \ldots \geq \mu_\usedim > 0$, consider the
associated ellipse
\begin{align}
  \label{EqnEllips}
\ellip \defn \Big\{ \theta \in \real^\usedim \mid \sum_{i=1}^\usedim
\frac{\theta_i^2}{\mu_i} \leq 1 \Big\}.
\end{align}
Suppose that we make observations of the form $y = \theta + \sigma g$.
Here $\theta \in \Ellipse$ is an unknown vector, whereas $g \sim N(0,
I_{\usedim})$ is a noise vector, and the variance level $\sigma^2$ is known.  For
a given vector $\thetastar \in \Ellipse$, our goal is to test the null
hypothesis $\theta = \thetastar$ versus the alternative $\theta \in
\Ellipse \backslash \{\thetastar\}$.

Under this formulation, it is not possible to make any non-trivial
assertion about the power of any test, since the alternative allows
for vectors $\theta$ that are arbitrarily close to $\thetastar$.  In
order to make quantitative statements, we need to exclude a certain
$\epsilon$-ball around $\thetastar$ from the alternative.  Doing so
leads to the notion of the \emph{minimax testing radius} associated
this composite decision problem.  This minimax formulation was
introduced in the seminal work of Ingster and
co-authors~\cite{ingster1987minimax,ingster2012nonparametric}; since
then, it has been studied by many other researchers
(e.g.,~\cite{ermakov1991minimax,Spokoiny1998testing,lepski1999minimax,lepski2000asymptotically,baraud2002non}).

More precisely, for a given radius $\epsilon > 0$, we consider the
compound testing problem specified by the observation model $y \sim
N(\theta, \sigma^2 I_\usedim)$, and the null and alternative
\begin{align}
\label{EqnMainTesting}
\Hyp_0: \theta = \thetastar \in \ellip ~~~\text{versus}~~~ \Hyp_1:
\|\theta - \thetastar\|_2 \geq \epsilon, \theta \in \ellip.
\end{align}
We thus have a sequence of testing problems indexed by $\epsilon > 0$,
and our goal is to calibrate their difficulty in terms of the noise
level $\sigma^2$, the null vector $\thetastar$, and the local geometry
of the ellipse $\Ellipse$ around $\thetastar$.  To be clear, all the
tests that we analyze in this paper are \emph{not} given knowledge of
$\epsilon$; rather, it is a unknown quantity that is used to titrate
the difficulty of the problem. See \autoref{fig:ellips} for an
illustration of the testing problem~\eqref{EqnMainTesting}.

\begin{figure}[H]
	\centering
	\widgraph{0.5\textwidth}{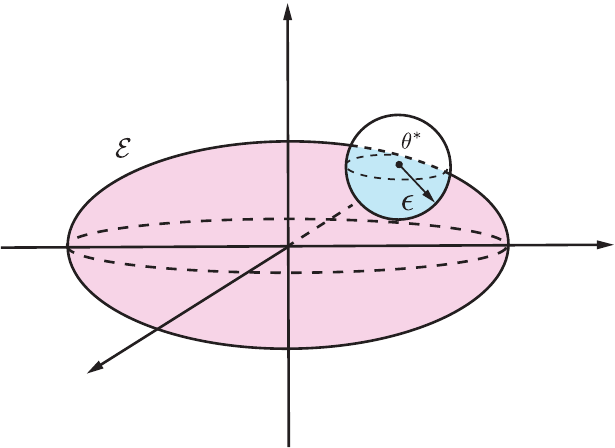}
	\caption{An illustration of testing
          problem~\eqref{EqnMainTesting}. }
	\label{fig:ellips}
\end{figure}

Letting $\psi: \real^\usedim \rightarrow \{0,1\}$ be any measurable
test function, we characterize its performance on the testing
problem~\eqref{EqnMainTesting} in terms of the sum of its (type I)
error under the null, and its (type II) error uniformly over the
alternative.  In particular, we define the \emph{uniform error}
\begin{align}
\label{EqnUniErr}
\UNIERR(\psi; \thetastar, \epsilon) \defn \Exs_{\thetastar} [\psi(y)]
+ \sup_{\theta\in \ellip, \|\theta - \thetastar\|_2 \geq \epsilon}
\Exs_\theta[1 - \psi(y)],
\end{align}
where $\Exs_\theta$ denotes expectation over $y$ under the $N(\theta,
\sigma^2 I_{\usedim})$ distribution.

For a given error level $\rho \in (0,1)$, we are interested in the
smallest setting of $\epsilon$ for which there exists some test with
uniform error at most $\rho$.  More precisely, we define
\begin{align}
\label{EqnDefnLocalMinimaxOPT}
\epsopt(\thetastar; \rho, \ellip) \defn \inf \Big \{ \epsilon \, \mid \,
\inf_{\psi} \; \UNIERR(\psi; \thetastar, \epsilon) \leq \rho \Big \},
\end{align}
a quantity which we call the \emph{$\thetastar$-local minimax testing
  radius}.  By definition, the local minimax testing radius
$\epsopt(\thetastar; \rho, \ellip)$ corresponds to the smallest separation
$\epsilon$ at which there exists \emph{some test} that distinguishes
between the hypotheses $\Hyp_0$ and $\Hyp_1$ in
equation~\eqref{EqnMainTesting} with uniform error at most $\rho$.
Thus, it provides a fundamental characterization of the statistical
difficulty of the hypothesis testing as a function of $\thetastar.$
Our results also have consequences for the more classical \emph{global
  minimax testing radius}, given by
\begin{align*}
\epsopt(\rho) \defn \sup_{\thetastar \in \Ellipse} \epsopt(\thetastar;
\rho, \ellip).
\end{align*}
As we discuss in the sequel, this global quantity has been
characterized for various Sobolev ellipses in past
work~\cite{ingster2012nonparametric}.

\subsection{Overview of our results}

Having set up the problem, let us now provide a high-level overview of
the main results of this paper.  We will show that the testing radius is
fully described by a purely geometric quantity, namely the
\emph{Kolmogorov $k$-width} of a set that localized at
$\thetastar$. Given a set $A$, the Kolmogorov width characterizes how
well the set is approximated by a $k$-dimensional linear subspace.  It
is known that the Kolmogorov width provides risk upper bounds for the
truncated series estimator
(e.g.~\cite{javanmard2012minimax,donoho1990minimax}) and it also turns
out to be a key quantity in problems such as density estimation, as
well as compressed sensing (see the
papers~\cite{hasminskii1990density,donoho2006compressed}).  We discuss
the definition of the Kolmogorov width and its property at length in
Section~\ref{sec:width}.

We will also show that the optimal test, as will be discussed in sequel, is given by a linear projection test that is based on
the projection of observation vector to a particular lower dimensional
subspace.  More discussions on its formulation and background can be
found in our Section~\ref{sec:LPT}.  We state our main results and
some of their important consequences in Section~\ref{sec:OPT}. 
Section~\ref{sec:main_proof} is devoted to the proofs of our main
theorems and corollaries.  We defer the auxiliary lemmas for our main
proofs to the appendices.


\section{Background}
\label{SecBackground}

Before proceeding to the statements of our main results, we introduce some
background on linear projection tests, as well as the notion of Kolmogorov
width.


\subsection{Linear projection tests}
\label{sec:LPT}

In this paper, we prove some upper bounds on $\epsopt(\thetastar;
\rho,\ellip)$ via a concretely defined class of tests, in particular
those based on linear projections.  Given an integer $k \in [\usedim]
\defn \{1, \ldots, \usedim \}$, let $\trunc{k}$ denote an orthogonal
projection operator, mapping any vector $v \in \real^\usedim$ to a
particular $k$-dimensional subspace.  Any projection $\trunc{k}$
defines a family of linear projection tests, indexed by the threshold
$\beta \geq 0$, of the following form
\begin{align}
\label{EqnLPT}
\psi_{\beta, \trunc{k}}(y) \defn \Ind \big[ \ltwo{\trunc{k} (y -
    \thetastar)} \geq \beta \big] \; = \; \begin{cases} 1 & \mbox{if
    $\|\trunc{k}(y - \thetastar)\|_2 \geq \beta$} \\ 0 &
  \mbox{otherwise.}
\end{cases}
\end{align}
We use $\testTr$ to denote the collection of all linear projection
tests
\begin{align*}
\testTr & \defn \Big \{ \psi_{\beta, \trunc{k}} \; \mid k \in
        [\usedim], \; \trunc{k} \in \mathcal{P}_k, \; \beta \in
        \real_+ \Big \},
\end{align*}
obtained by varying the dimension $k \in [\usedim]$, the projection
$\trunc{k}$ over the space $\mathcal{P}_k$ of all $k$-dimensional
orthogonal projections, and the threshold $\beta \geq 0$.

Given the family $\testTr$, we then define
\begin{align}
\label{EqnDefnLocalMinimaxLPT}
\epslpt(\thetastar ; \rho, \ellip) \defn \inf \Big \{ \epsilon \, \mid \,
\inf_{\psi \in \testTr} \; \UNIERR(\psi; \thetastar, \epsilon) \leq
\rho \Big \}.
\end{align}
a quantity which we call 
\emph{$\thetastar$-local LPT testing radius}.  When the error level 
$\rho$ is clear from the context, we adopt the shorthand notation
$\epslpt(\at ; \ellip)$ and $\epsopt(\at ; \ellip)$, with the implicit
understanding that all of our analysis depends on the choice of
$\rho$.


\subsection{Kolmogorov width}
\label{sec:width}

Our characterization of testing difficulty involves a classical
geometric notion known as the Kolmogorov width, which we introduce
here; we refer the reader to the book by Pinkus~\cite{pinkus2012n} for
more details.

For a given compact set $C \subset \real^\usedim$ and integer $k \in
[\usedim]$, the Kolmogorov $k$-width is a measure of how well the set
can be approximated by a $k$-dimensional subspace.  More precisely,
recalling that $\LPset{k}$ denotes the set of all $k$-dimensional
orthogonal linear projections, the \emph{Kolmogorov $k$-width},
denoted by $\kwidth_k$, is given by
\begin{align}
  \label{EqnKolmogorov}
\kwidth_k (C) = \min_{\trunc{k} \in \mathcal{P}_k} \max_{\theta \in C}
\|\theta - \trunc{k} \theta\|_2.
\end{align}
Any $\trunc{k}$ achieving the minimum in the saddle point
problem~\eqref{EqnKolmogorov} is said to be an \emph{optimal
  projection} for $\kwidth_k (C)$.  By definition, the Kolmogorov
widths are non-increasing as a function of the integer $k$---in
particular, we have
\begin{align*}
\rad(C) = \kwidth_0(C) \geq \kwidth_1(C) \geq \ldots \geq \kwidth_d(C) = 0,
\end{align*}
where $\rad(C) \defn \max_{\theta \in C} \|\theta\|_2$ is the diameter
of the compact set $C$.


\section{Main results and their consequences}
\label{sec:OPT}

We now turn to the statement of our main results, along with a
discussion of some of their consequences.


\subsection{Upper bound on the local minimax testing radius} 
\label{sec:ub_kl}

In this section, we establish an upper bound on the minimax testing
radius for problem~\eqref{EqnMainTesting} localized to a given vector
$\thetastar$.  In order to do so, we study the behavior of the linear
projection tests previously described in Section~\ref{sec:LPT}, It
turns out that LPT testing radius $\epslpt(\thetastar; \ellip)$ can be
characterized by the Kolmogorov $k$-width---as defined in
expression~\eqref{EqnKolmogorov}---of a particular set.

Before stating the main result, let us first introduce some notation.
For each $\epsilon > 0$, we define the \emph{upper critical dimension}
\begin{align}
  \label{EqnK-upper}
  \upk(\epsilon, \thetastar, \ellip) \defn \arg \min_{1 \leq k \leq
    \usedim} \Big \{ \kwidth_k(\ellip_{\thetastar} \cap
  \Ball(\epsilon)) \leq \frac{\epsilon}{\sqrt{2}} \Big \}.
\end{align}
Here set $\ellip_{\thetastar} \defn \{\theta - \at \mid \theta\in \ellip\}$
for the ellipse defined in expression \eqref{EqnEllips} and 
$\Ball(\epsilon) \defn \{v \in \real^\usedim \mid \ltwo{v}\leq \epsilon\}.$
For any fixed $\epsilon > 0$, the Kolmogorov width
$\kwidth_k(\ellip_{\thetastar} \cap \Ball(\epsilon))$ is non-increasing
with index $k$ and it goes to zero as $k$ approaches $\usedim$ (where
$\usedim$ can grow as well, then we have an infinite dimensional
ellipse) and $\kwidth_\usedim(\ellip_{\thetastar} \cap \Ball(\epsilon)) = 0$.  
Therefore, the critical dimension is always well defined.

This dimension is used to define the optimal linear projection test at
$\thetastar$.  We also use it to define for each error level $\rho \in
(0,1/2]$, the \emph{upper critical radius} 
\begin{align}
\label{EqnRadCritU}
  \epscritu(\thetastar ; \rho, \ellip) & \defn \inf \{ \epsilon \mid
  \epsilon \geq \frac{8}{\sqrt{\rho}} \sigma^2
  \frac{\sqrt{\upk(\epsilon, \thetastar, \ellip)}}{\epsilon} \Big \}.
\end{align}
For each choice of error level $\rho$, this radius defines the
separation between null and alternative that can be distinguished with
error controlled at level $\rho$.  From the fact that the function
$\epsilon \mapsto \upk(\epsilon, \thetastar, \ellip)$ is
non-increasing, it follows that $\epscritu(\thetastar ; \rho, \ellip)$
always exists, and is guaranteed to be unique and strictly positive.
More details can be found in Section~\ref{AppSolution}.

\noindent With these two ingredients, we are ready to state our upper
bound for the LPT testing radius localized to $\thetastar$:
\begin{theos}
\label{ThmProjUB}
For any vector $\thetastar \in \ellip$ and any error level $\rho \in
(0,1/2]$, we have
\begin{align}
\label{EqnProjUB-KM}
\inf_{\psi \in \testTr} \; \UNIERR(\psi; \thetastar, \plaindel) \leq \rho
\qquad \text{ for all } \plaindel \geq \epscritu(\thetastar ; \rho, \ellip).
\end{align}
\end{theos}
\noindent We prove this theorem in Section~\ref{AppThmUB}.   


Recall the definition of the $\thetastar$-localized minimax radius
$\epsopt$ and $\thetastar$-localized LPT radius $\epslpt$ from
equations~\eqref{EqnDefnLocalMinimaxOPT}
and~\eqref{EqnDefnLocalMinimaxLPT}, respectively. In terms of this
notation, Theorem~\ref{ThmProjUB} guarantees that
\begin{align*}
  \epsopt(\thetastar; \rho, \Ellipse) \; \stackrel{(i)}{\leq} \;
  \epslpt(\thetastar; \rho, \Ellipse) \; \stackrel{(ii)}{\leq} \;
  \epscritu(\thetastar; \rho, \ellip).
\end{align*}
To be clear, the first inequality (i) is a trivial consequence of the
definitions, whereas inequality (ii) is the substance of
Theorem~\ref{ThmProjUB}.   


\paragraph{Structure of optimal linear test:}  Although we have stated
the theorem in an existential way, our analysis also gives an explicit
form of the test that achieves the error
guarantee~\eqref{EqnProjUB-KM}.  The construction of this test
consists of three steps:
\begin{itemize}
\item For a given error level $\rho$, first compute the critical
  radius $\epscritu(\thetastar ; \rho, \Ellipse)$ using
  expression~\eqref{EqnRadCritU}.
\item Second, using equation~\eqref{EqnK-upper}, compute the induced
  critical upper dimension
  \begin{align*}
 \upkstar \defn \upk(\epscritu(\thetastar ; \rho, \ellip), \thetastar,
 \ellip),
  \end{align*}
  as well as the associated projection matrix $\Pi_{\upkstar}$ that
  defines the Kolmogorov width indexed by $\upkstar$.
\item  Third, apply the linear projection test
\begin{align}
\label{EqnOptTest}
y \mapsto \Ind \Big[ \ltwo{\Pi_{\upkstar}(y - \thetastar)}^2 \geq
  \sigma^2(\upkstar + \sqrt{4\upkstar/\rho}) \Big].
\end{align}
\end{itemize}
The proof of the theorem is based on analyzing this linear
test~\eqref{EqnOptTest} directly and explicitly controlling its type I
and type II error.

\noindent In order to illustrate the use of Theorem~\ref{ThmProjUB},
let us consider a very simple example.

\begin{exas}[{\bf Circular constraints}]
  Suppose that the ellipse $\Ellipse$ is simply a circle in
  $\real^\usedim$, say with $\mu_j = \usedim$ for all $j = 1,
  \ldots, \usedim$. Let us consider the case of testing at zero (so
  that $\thetastar = 0$) with a noise level $\sigma \in (0,1)$.  For
  any $\epsilon \in (0, \sqrt{\usedim})$, we claim that the upper
  critical dimension~\eqref{EqnK-upper} is equal to $\usedim$.
  Indeed, for any projection matrix of dimension $k < \usedim$, there
  is at least one dimension that is missing, and this dimension
  contains a vector of Euclidean length $\sqrt{\usedim}$ inside the
  sphere.  Consequently, solving for the upper critical radius from
  equation~\eqref{EqnRadCritU} yields that $\epscritu^2(0; \ellip) =
  \frac{8}{\sqrt{\rho}} \sigma^2 \sqrt{\usedim}$.
  \end{exas}

Note that this rate matches the minimax testing radius for a subspace
of dimension $\usedim$.  This calculation makes sense, because for
this particular example---that is, testing at the zero vector for a
circle with the chosen scalings of $\mu_j$ and $\sigma$---the geometry
and boundary of the circle actually plays no role.

In Section~\ref{SecConsequences}, we consider some more substantive
applications of Theorem~\ref{ThmProjUB}, and in particular,
demonstrate that the local minimax testing radius varies substantially
as a function of the location of $\at$ within the ellipse.


\subsection{Lower bound on the local minimax testing radius} 
\label{sec:lw_kl}

Thus far, we have derived an upper bound for the localized testing
radius for a particular procedure---namely, the LPT.  Of course, it is
of interest to understand when the LPT is actually an optimal test,
meaning that there is no other test that can discriminate between the
null and alternative for smaller separations.  In this section, we use
information-theoretic methods to prove a lower bound on the localized
minimax testing radius $\epsopt(\thetastar; \ellip)$.  Whenever the
radius $\epslpt(\thetastar; \ellip)$ achievable by a linear projection
test matches the lower bound for every $\thetastar$, we can claim the
linear projection tests are \emph{locally minimax optimal}.


For the testing problem \eqref{EqnMainTesting}, it turns out that the
local minimax testing radius $\epsopt(\thetastar; \ellip)$ can also be
lower bounded via the Kolmogorov $k$-width, as previously defined in
equation~\eqref{EqnKolmogorov}, of a particular set.  In order to
state this lower bound, we fix a triplet of positive constants $a,b,c$
linked by the relations
\begin{align}
\label{EqnFrog}
a > 3b, \mbox{ and } \ccon = \ccon(a,b) \defn \frac{b}{8\sqrt{2}}-
\frac{\sqrt{a^2-9b^2}}{12\sqrt{2}} > 0.
\end{align}
Recall the recentered ellipse \mbox{$\ellip_{\thetastar} \defn
  \{\theta - \at \mid \theta\in \ellip\}$}, and Euclidean ball of radius $\epsilon$ centered at zero $\Ball(\epsilon)
\defn \{v \in \real^\usedim \mid \ltwo{v}\leq \epsilon\}$.

For our upper bound, the upper critical dimension $\upk$ from
equation~\eqref{EqnK-upper} played a central role.  As the lower
analogue to this quantity, let us define the \emph{lower critical
  dimension}
\begin{align}
\label{EqnK-lower}
\lwk(\epsilon, \thetastar, \ellip) \defn \arg \min_{1\leq k \leq
  \usedim} \braces*{ \kwidth_k(\ellip_{\thetastar} \cap
  \Ball(a\epsilon)) \leq 3b \epsilon},
\end{align}
where the fixed choice of constants $(a,b)$ should be understood
implicitly.  We then define the \emph{lower critical radius}
\begin{align}
\label{EqnRadCritL}
  \epscritl(\thetastar; \ellip) & \defn \sup \Big \{ \epsilon \mid
  \epsilon \leq \frac{1}{4}
  \sigma^2 \frac{\sqrt{\lwk(\epsilon, \thetastar, \ellip)}}{\epsilon} \Big \},
\end{align}
From the fact that the function $\epsilon \mapsto \lwk(\epsilon,
\thetastar, \ellip)$ is non-increasing, it follows that $\epscritl$ is
well-defined and unique.  More details can be found in
Appendix~\ref{AppSolution}.

We prove a lower bound as a function of the quantity
$\epscritl(\thetastar; \ellip)$ and also a second quantity, one which
measures its proximity to the boundary of the ellipse.  More
precisely, for a given vector $\thetastar$ and constant $a > 0$, we
define the mapping $\Rfun: \real_+ \rightarrow \real_+$ via
\begin{align}
\label{EqnRfun}
  \Rfun(\delta) & = \begin{cases} 1 & \mbox{if $\delta > \|\thetastar\|_2/a$} \\
    1 \wedge \min \Big \{ r \geq 0 \, \mid a^2 \delta^2 \leq
    \sum_{i=1}^\usedim \frac{r^2}{(r+\mu_i)^2} (\at_i)^2 \Big \} & \mbox{otherwise.}
  \end{cases}
\end{align}
As shown in Appendix~\ref{AppRfun}, this mapping is well-defined, and
has the limiting behavior $\Rfun(\delta) \rightarrow 0$ as $\delta
\rightarrow 0^+$.  Let us denote $\Rfun^{-1}(x)$ as the largest
positive number of $\delta$ such that $\Rfun(\delta) \leq x$. Note
that by this definition, we have $\InvRfun(1) = \infty$.

\noindent We are now ready to state our lower bound for the optimal
testing radius localized to $\at$:
\begin{theos}
\label{ThmLowerBound}
Consider constants $a,b, c\in (0,1)$ satisfying the
conditions~\eqref{EqnFrog}.  Then for any $\thetastar \in \Ellipse$,
we have
\begin{align}
\label{EqnProjLB-KM}
\inf_{\psi} \; \UNIERR(\psi; \thetastar, \plaindel) \geq \frac{1}{2}
\qquad \text{ whenever } \quad 
\plaindel \, \leq \, \ccon \min \Big\{ \epscritl(\thetastar; \ellip), 
~~ \Rfun^{-1} \big( (\enorm{\thetastar}^{-1} - 1)^2 \big) \Big \}.
\end{align}
\end{theos}
\noindent See Section~\ref{AppThmKLB} for the proof of this theorem. 

\paragraph{Remarks:}  Regarding concrete choices of the constants $(a,b,c)$
that underlie this theorem, if we set $a = \frac{\sqrt{43}}{12}$ and
$b=\frac{1}{4}$, then $\ccon = \frac{1}{288\sqrt{2}}$. Since our main interest
is to understand the scaling of the testing radius with respect to
$\sigma$ and the geometric parameters of the problem, we have made no
efforts to obtain the sharpest constants in the theorem statement.

Second, to clarify the role of the term $\Rfun^{-1} \big(
(\enorm{\thetastar}^{-1} - 1)^2 \big)$, it is not of primary interest
and serves to measure whether or not $\thetastar$ is sufficiently far
inside the ellipse.  Concretely, if we assume that $\enorm{\at} \leq
1/2$, then $(\enorm{\thetastar}^{-1} - 1)^2 \geq 1$ therefore
$\InvRfun \big ( (\enorm{\thetastar}^{-1} - 1)^2 \big) = \infty$ which
means that
\begin{align*}
\min \Big\{\epscritl(\thetastar; \ellip), ~~ \Rfun^{-1} \big(
(\enorm{\thetastar}^{-1} - 1)^2 \big) \Big \} = 
\epscritl(\thetastar; \ellip).
\end{align*}

A final remark on the proof of Theorem~\ref{ThmLowerBound}: it is
based on constructing a prior distribution on the alternative $\Hyp_1$
that is independent of observation $y$.  This distribution is
supported on precisely those points in $\Hyp_1$ that are relatively
hard to distinguish from $\Hyp_0$.  Then the testing error can be
lower bounded through controlling the total variation (TV) distance between
two marginal likelihood functions.


\subsection{Some consequences of our results}
\label{SecConsequences}

One useful consequence of Theorems~\ref{ThmProjUB}
and~\ref{ThmLowerBound} is in providing a sufficient condition for the
optimality of the LPT.  In particular, suppose that the upper and
lower dimensions defined in equations~\eqref{EqnK-upper}
and~\eqref{EqnK-lower} differ only by constant pre-factors, which we
write as
\begin{align}
\label{EqnOptCond}
\upk(\epsilon, \thetastar, \ellip) \asymp \lwk(\epsilon; \thetastar,
\ellip).
\end{align}
Under this condition, it follows that the critical upper and lower
radii match---namely, that $\epscritu(\thetastar; \ellip) \asymp
\epscritl(\thetastar; \ellip)$.  Theorem~\ref{ThmProjUB} then
guarantees that the linear projection test at $\thetastar$ is locally
minimax optimal.

As previously discussed, the local optimality studied here provides a
finer-grained distinction than the usual notion of global minimax
optimality.  All of the rates depend on the vector $\thetastar$ via
the set $\Ellipse_\at \cap \Ball(\epsilon)$, whose shape changes as
$\at$ moves across the ellipse.  The resulting changes in the local
minimax radii can be quite substantial, as we now illustrate with some
concrete examples.


\subsubsection{Testing at zero}
\label{sec:zero}

We begin our exploration by considering the testing
problem~\eqref{EqnMainTesting} with $\thetastar = 0$.  In order to
characterize the $\at$-localized minimax testing
radius~\eqref{EqnDefnLocalMinimaxOPT}, we define the functions
\begin{align}
\label{EqnDefKzero}
\newupk(\delta; \ellip) \defn \arg \max_{1\leq k \leq \usedim} \{
\mu_k \geq \frac{1}{2}\delta^2 \}, ~\text{ and }~ \newlwk(\delta; \Ellipse)
\defn \arg \max_{1\leq k \leq \usedim} \{ \mu_{k+1} \geq \frac{9 }{16}\delta^2
\}.
\end{align}
In order to ensure that both functions are well-defined, we restrict
our attention to $\delta$ in the interval $\big(0,
\min\{\sqrt{2\mu_1}, \frac{4}{3} \sqrt{\mu_2}\} \big)$.  We now state
a consequence of Theorems~\ref{ThmProjUB} and~\ref{ThmLowerBound} in
terms of these quantities:
\begin{cors}
\label{CorZero}
Given an error level $\rho \in (0,1)$, the local minimax testing
radius $\epsopt(\at; \ellip)$ at $\thetastar = 0$ is bounded as
\begin{align}
\label{EqnMinuet}
 \sup \Big \{ \delta \, \mid \, \delta \leq \frac{1}{4}\sigma^2
 \frac{\sqrt{\newlwk(\delta; \Ellipse)}}{\delta} \Big \} \; \leq \;
 \epsopt(0; \ellip) \; \leq \; \inf \Big \{ \delta \, \mid \, \delta \geq
 \frac{8}{\sqrt{\rho}} \sigma^2 \frac{\sqrt{\newupk(\delta;
     \Ellipse)}}{\delta} \Big \},
\end{align} 	
where we only consider $\delta \leq \min\{\sqrt{2\mu_1}, \frac{4}{3}
\sqrt{\mu_2}\}$.
\end{cors}
\noindent We prove Corollary~\ref{CorZero} in
Section~\ref{AppCorZero}.

The quantities defined on the left and right hand sides of
expression~\eqref{EqnMinuet} exist due to the fact that
$\newupk(\delta; \ellip)$ and $\newlwk(\delta; \Ellipse)$ are
non-increasing functions of $\delta.$ Putting inequalities in
\eqref{EqnMinuet} together, whenever $\newupk$ and $\newlwk$ differ
only by constant factors, we have a tight characterization of the
testing radius $\epsopt(0; \ellip)$ localized to $\thetastar = 0$.

\vspace*{0.5cm}

In order to illustrate, we now apply Corollary~\ref{CorZero} to some
special cases of the ellipse $\ellip$, and see how the minimax testing
radius localized to zero varies as a function of the ellipse
parameters.  In particular, let us consider two broad classes of decay
rates for the ellipse parameters $\{ \mu_j \}_{j}$.

\paragraph{$\gamma$-exponential decay:} 

For some $\gamma > 0$, suppose the scalars $\{ \mu_j \}_{j}$ satisfies a decay
condition of the form $\mu_j = c_1\exp(-c_2 j^\gamma)$, where
$c_1,c_2$ are universal constants.  Such exponential decay rates arise
the problem of non-parametric regression testing
(Example~\ref{ExaRegression}), in particular for reproducing kernel
Hilbert spaces based on Gaussian kernels.  From
expression~\eqref{EqnDefKzero}, we know $\upk = C
\log^{1/\gamma}(1/\epsilon)$ and $\kb = c\log^{1/\gamma}(1/\epsilon)$
and solving the two inequalities from equation~\eqref{EqnMinuet}
yields
\begin{align}
\label{EqnExpZeroRate}
\epsopt^2(0) \asymp \sigma^2 (\log(\frac{1}{\sigma^2}))^{1/2\gamma} .
\end{align}
Here $\asymp$ denotes equality up to a factor involving constants and
terms of the order $\log \log (1/\sigma^2)$.


\paragraph{$\alpha$-polynomial decay:}

Now suppose that the ellipse parameters $\{ \mu_j \}_j$ satisfy a
decay condition of the form $\mu_j = c_1 j^{-2\alpha}$ for some
$\alpha > 1/2$.  As previously discussed, polynomial decay of this
type arises in ellipses that are induced by Sobolev function spaces of
smoothness $\alpha$.  From expression~\eqref{EqnDefKzero}, we know
$\upk = C \epsilon^{-1/\alpha}$ and $\kb = c \epsilon^{-1/\alpha}$ and
solving the two inequalities in equation~\eqref{EqnMinuet} yields
\begin{align}
\label{EqnSobZeroRate}
\epsopt^2(0) \asymp (\sigma^2)^{\frac{4\alpha}{1+ 4\alpha}}.
\end{align}
Here our notation $\asymp$ denotes equality up to constants
independent of $(\sigma, \usedim)$, so that
inequality~\eqref{EqnSobZeroRate} provides a characterization of the
testing radius localized at zero that is tight up to constant factors.
It is worth noticing that that rate~\eqref{EqnSobZeroRate} matches the
(non-localized) globally minimax optimal testing radius established by
Ingster~\cite{ingster93a,ingster93b,ingster93c}.  Thus, testing at
zero is as hard as the hardest testing problem in the ellipse.
Intuitively, this correspondence makes sense because the ellipse
contains a relatively large volume around zero, meaning that the
compound alternative contains a larger number of possible vectors that
need to be distinguished from the null.

Ingster~\cite{ingster93a} proves the lower bound by the direct
analysis of a $\chi^2$-type test, which is then shown to be optimal.
For the estimation analogue of this problem, Ibragimov and
Khasminskii~\cite{ibragimov1978,ibragimov2013statistical} showed that
the global minimax rate for estimation scales as
\mbox{$\epsilon_{\text{est}} \asymp
  (\sigma^2)^{2\alpha/(1+2\alpha)}$}. This slower rate confirms
the intuition that testing is an easier problem than estimation:
we can detect a vector of length much smaller than the accuracy
at which we can estimate.

\paragraph{Remark:}  As a final comment,
Corollary~\ref{CorZero} can be generalized to the $\ell_p$-ellipse
with $p \geq 2$, given by \mbox{$\Ellipse_p \defn \{\theta \in
  \real^\usedim \mid \sum_{i=1}^\usedim \frac{\theta_i^p}{\mu_i} \leq
  1\}$}.  Let us again consider testing at $\thetastar = 0.$ Using the
same technique to prove Corollary~\ref{CorZero}, it is easy to check
that inequalities in equation~\eqref{EqnMinuet} still hold with the
definitions of $\newupk$ and $\newlwk$ in
expression~\eqref{EqnDefKzero} replaced by
\begin{align*}
  \newupk(\delta; \ellip) \defn \arg \max_{1 \leq k \leq \usedim} \big\{
  \mu_{k+1}^{2/p} \geq \frac{1}{2}\delta^2 \big\}, \quad \mbox{and} \quad
  \newlwk(\delta; \ellip) \defn \arg \max_{1 \leq k \leq \usedim} \big\{
  \mu_k^{2/p} \geq \delta^2\big\}.
\end{align*}


\subsubsection{Testing at extremal vectors} 

In the previous section, we studied the behavior of the minimax
testing radius for $\at = 0$, and recovered some known results from
the classical literature.  In this section, we study some non-zero
cases of the vector $\at$.  For concreteness, we restrict our
attention to vectors that are non-zero some coordinate $s \in
[\usedim] = \{1, \ldots, \usedim\}$, and zero in all other
coordinates.  Even for such simple vectors, our analysis reveals some
interesting scaling of the minimax testing radius that is novel.

Given an integer $s \in [\usedim]$ and a positive scalar $\delta$, we
define
\begin{align}
\label{EqnDefKcorner}
\newupk(\delta;s) & \defn \arg \max_{1 \leq k \leq \usedim} \Big
\{\mu_k^2 \geq \frac{1}{64}\delta^2\mu_s \Big \}, ~\text{ and }~
\newlwk(\delta; s) \defn \arg \max_{1 \leq k \leq \usedim} \Big
\{\mu^2_k \geq \delta^2 \mu_s \Big \}.
\end{align}
So as to ensure that these functions are well-defined, we restrict our
attention to $\delta$ in the interval $\big(0, \mu_1/\sqrt{\mu_s}
\big)$.  By definition, we have $\newupk(\delta;s) \geq
\newlwk(\delta;s)$.

Before stating our corollary, let us define the upper and lower
critical radii as follows
\begin{subequations}
\begin{align}
\label{EqnRadKcorner}
  \newepscritu(s, \ellip) &= \inf \Big\{\delta>0 \mid \delta \geq
  \frac{8}{\sqrt{\rho}}
  \sigma^2\frac{\sqrt{\newupk(\delta;s)}}{\delta} \Big\} \\ \text{ and
  }~~\newepscritl(s, \ellip) &= \sup \Big\{\delta>0 \mid \delta \leq
  \frac{1}{4} \sigma^2 \frac{\sqrt{\newlwk(\delta;s)}}{\delta} \Big\}.
\end{align}
\end{subequations}
Using the fact that $\rho\in (0,1/2]$ and non-increasing functions 
satisfy $\newupk(\delta;s) \geq \newlwk(\delta;s)$, we have $\newepscritl(s, \ellip) \leq
  \newepscritu(s, \ellip)$.  Again these two quantities do not depend
  on the $\at$ that we are testing at.  We then define the set
\begin{align*}
  C(s; \Ellipse) & \defn \Big \{ \theta \in \real^\usedim \mid
  \theta_{j} = 0 \quad \mbox{for all $j \neq s$, and} \quad \theta_s
  \in \big [\sqrt{\mu_s} - \newepscritu(s, \ellip), \sqrt{\mu_s}
    -\newepscritl(s, \ellip) \big] \Big \}.
\end{align*}
We are ready to state our corollary:
\begin{cors}
\label{CorCorner}
Consider any positive integer $s$ such that $\newepscritu(s, \ellip)
\leq \sqrt{\mu_s}$.  Then for any testing error level $\rho > 0$, we
have
\begin{align}
\label{EqnCorner}
  \newepscritl(s, \ellip) \,\leq\, \epsopt(\at; \ellip) \,\leq\,
  \newepscritu(s, \ellip) \qquad \mbox{for all $\thetastar \in C(s; \Ellipse)$.}
\end{align}
\end{cors}
\noindent 
See Section~\ref{AppCorCorner} for the proof of
Corollary~\ref{CorCorner}.

Corollary~\ref{CorCorner} characterizes the optimal testing radius for
vectors that are sufficiently close to the boundary of the ellipse, In
order to compare with the results of Section~\ref{sec:zero} for
testing at zero, let us apply Corollary~\ref{CorCorner} to some
special examples of the ellipse $\ellip$, and see how the testing
radius varies as we vary the index $s \in [\usedim]$.
For these comparisons, we restrict our attention to ellipses
for which the parameters $\{\mu_j\}$ decay at a polynomial rate.


\paragraph{$\alpha$-Polynomial decay:}

Suppose that the scalars $\{ \mu_j \}_{j}$ satisfy a decay condition of the form
\mbox{$\mu_j = c_1 j^{-2\alpha}$} for some $\alpha > 1/2$.  Recalling
the definition~\eqref{EqnDefKcorner}, some calculations yield
\begin{align*}
  \frac{\mu_{\mind_u}^2}{\mu_s} & = \frac{\mind_u^{-4 \alpha}}{s^{-
      2\alpha}} ~\geq~ \frac{\delta^2}{64} ~\Longrightarrow~ \mind_u =
  \lfloor (64 s^{2\alpha})^{\frac{1}{4\alpha}}
  \delta^{-\frac{1}{2\alpha}} \rfloor, \qquad \mbox{and} \\
  \frac{\mu_{\mind_\ell}^2}{\mu_s} & = \frac{\mind_\ell^{-4
      \alpha}}{s^{- 2\alpha}} ~\geq~ \delta^2 ~\Longrightarrow~
  \mind_\ell = \lfloor s^{\frac{1}{2}} \delta^{-\frac{1}{2\alpha}}
  \rfloor.
\end{align*}
Solving the fixed point equations yields
\begin{align*}
  \newepscritu(s, \ellip) =
  \Big(8^{\frac{4\alpha+1}{4\alpha}}\rho^{-\frac{1}{2}}s^{\frac{1}{4}}
    \sigma^2 \Big)^{\frac{4\alpha}{8\alpha+1}}
\qquad \text{and }~~ \newepscritl(s, \ellip) = \Big(\frac{1}{4}
  s^{\frac{1}{4}} \sigma^2 \Big)^{\frac{4\alpha}{8\alpha+1}}.
\end{align*}
If we omit the dependence on the level $\rho$ and constants, the
optimal testing radius is given by
\begin{align}
\label{EqnSobCornerRate}
\epsopt^2(\thetastar; \ellip)  \asymp (\sigma^2 s^{1/4})^{\frac{8\alpha}{1 + 8\alpha}}
\qquad 
\text{and }~~
\mind_u,\mind_\ell \asymp (\sigma^{-2}s^{2\alpha})^{\frac{2}{8\alpha+1}}
\end{align}
whenever\footnote{The constraint on $s$ arises from the fact that
  $\newepscritu(s, \ellip) \leq \sqrt{\mu_s}$.} $s \lesssim
(\sigma^2)^{-2/(4\alpha+1)}$.  As the choice of $s$ is varied,
inequality~\eqref{EqnSobCornerRate} actually provides us with a range
of different results.
\begin{itemize}
\item[(a)] First considering $s = 1$, suppose that $\at$ is zero at
  every coordinate except the first: we then have
\begin{align}
  \label{EqnFastRate}
  \epsopt^2(\thetastar;\ellip) \asymp (\sigma^2)^{\frac{8\alpha}{1 +
      8\alpha}},
\end{align}
where $\asymp$ means equal up to a constant that is independent of
problem parameters $\usedim$ and $\sigma$.  To the best of our
knowledge, this is a novel result on ellipse testing.  It shows that
that the local minimax rate at this extremal vector is very different
from the corresponding rate at zero, which as we have established in
our Section~\ref{sec:zero} is equal $(\sigma^2)^{\frac{4\alpha}{1 +
    4\alpha}}$.
\item[(b)] On the other hand, let us consider a sequence (indexed by
  $\sigma$) of the form $s = (\sigma^2)^{-\beta}$ for some $\beta \in
  \Big( 0, \frac{2}{(4\alpha+1)} \Big)$. Over such a sequence,
  inequality~\eqref{EqnSobCornerRate} shows that the minimax
  testing radius scales as
\begin{align}
\epsopt^2(\thetastar;\ellip) \asymp
(\sigma^2)^{\frac{2\alpha(4-\beta)}{1 + 8\alpha}}.
\end{align}
As $\beta$ ranges from $0$ to $\frac{2}{(4\alpha+1)}$, the testing
radius $\epsopt^2(\thetastar; \ellip)$ ranges from
$(\sigma^2)^{\frac{4\alpha}{1 + 4\alpha}}$ to
$(\sigma^2)^{\frac{8\alpha}{1 + 8\alpha}}$.
\end{itemize}


\section{Proofs of main results}
\label{sec:main_proof}

We now turn to the proofs of our main results, with the 
proof of Theorems~\ref{ThmProjUB} and \ref{ThmLowerBound}
given in Sections~\ref{AppThmUB} and \ref{AppThmKLB} 
respectively, followed by the proofs of 
Corollary~\ref{CorZero} and \ref{CorCorner} in 
Sections~\ref{AppCorZero} and \ref{AppCorCorner}, respectively.
In all cases, we defer the proofs of certain more technical lemmas to the appendices.


\subsection{Proof of Theorem~\ref{ThmProjUB}}
\label{AppThmUB}

From our discussion following the statement of
Theorem~\ref{ThmProjUB}, recall that the optimal linear
test~\eqref{EqnOptTest} is based on the statistic
\begin{align}
  \label{EqnSunset}
  \psi(y) \defn \Ind \big[ \ltwo{\Pi_{\upkstar}(y - \thetastar)}^2
    \geq \sigma^2(\upkstar + \sqrt{4\upkstar/\rho}) \big],
\end{align}
where the critical upper dimension $\upkstar \defn
\upk(\epscritu(\thetastar ; \rho, \ellip), \thetastar, \ellip)$ is
computed from equation~\eqref{EqnK-upper}.

From here onwards, we adopt $\epsilon$ and $k$ as convenient
shorthands for the quantities $\epscritu$ and $\upkstar$ respectively.
We let $\trunc{k}$ denote the projection matrix that defines the
Kolmogorov $k$-width of $\ellip_\thetastar \cap
\Ball(\epsilon)$---that is
\begin{align}
\label{EqnAllegro}
\trunc{k} \defn \arg \min_{\trunc{k} \in \mathcal{P}_k} \max_{\Delta
  \in \ellip_{\thetastar} \cap \Ball(\epsilon)} \|\Delta - \trunc{k}
\Delta\|_2^2.
\end{align}

\paragraph{Controlling the type I error:} We first show that the
test statistic~\eqref{EqnSunset} has type I error at most
$\rho/2$. Under the null hypothesis $\Hyp_0$, we have $y \sim
N(\thetastar, \sigma^2 I_{\usedim})$.  Since $\trunc{k}$ is an orthogonal
projection matrix, the rotation invariance of the Gaussian
distribution guarantees that
\begin{align*}
T_k & \defn \|\trunc{k} (y - \thetastar)\|_2^2 \stackrel{d}{=}
\|\tilde{w}\|_2^2, \qquad \mbox{where $\tilde{w} \sim \NORMAL(0,
  \sigma^2 \Ind_k)$.}
\end{align*}
Consequently, the statistic $T_k/\sigma^2$ follows a
$\chi_k^2$-distribution, and hence
\begin{align}
\label{EqnTailtypeIGen}
\Prob_{0} \big[ \psi(y) = 1 \big] & = \Prob_{\thetastar} \big[T_k \geq
  \sigma^2(k + \sqrt{4k/\rho}) \big] \leq \frac{\rho}{2},
\end{align}
where the last step follows from Chebyshev's inequality.


\paragraph{Controlling the type II error:} Under the alternative
hypothesis $\Hyp_1$, we have $y \sim N(\theta, \sigma^2 I_{\usedim})$
for some vector $\theta \in \Ellipse$ with $\|\theta - \thetastar\|_2
\geq \epsilon$.  Since $\trunc{k}$ is a linear operator, we have
\begin{align*}
T_k = \|\trunc{k}(y - \thetastar)\|_2^2 = \|\trunc{k} w +
\trunc{k}(\theta - \thetastar)\|_2^2, \qquad \mbox{where $w \sim N(0,
  \sigma^2 I_{\usedim})$.}
\end{align*}
Using the fact that Gaussian distribution is rotation invariant, the
statistic $T_k/\sigma^2$ follows a non-central $\chi^2_k(c_0)$
distribution, defined by the recentering $c_0 = \|\trunc{k}(\theta
-\thetastar) \|_2^2$.

For each vector $\theta \in \Ellipse$ such that $\|\theta -
\thetastar\|_2\geq \epsilon$, we claim that
\begin{align}
\label{EqnAdagio}
c_0 = \|\trunc{k}(\theta -\thetastar) \|^2_2 \geq \min_{u \in \ellip,
  \|u - \thetastar\|_2 = \epsilon} \{\|\trunc{k}(u -\thetastar) \|^2_2
\}.
\end{align}
We take inequality~\eqref{EqnAdagio} as given for the moment,
returning to prove it at the end of this section.  Since $\trunc{k}$
is an orthogonal projection onto a subspace, the Pythagorean theorem
guarantees that
\begin{align*}
\|\trunc{k}(u -\thetastar) \|^2_2 = \|(u -\thetastar)\|_2^2 - \|(u
-\thetastar) - \trunc{k}(u -\thetastar) \|^2_2.
\end{align*}
Consequently, we have
\begin{align*}
c_0 = \|\trunc{k}(\theta - \thetastar) \|^2_2 & \stackrel{(i)}{\geq}
\min_{u \in \ellip, \|u - \thetastar\|_2 = \epsilon} \, \left\{ \|(u
-\thetastar)\|_2^2 - \|(u - \thetastar) - \trunc{k}(u -\thetastar)
\|^2_2 \right\} \\
&= \epsilon^2 - \max_{u \in \ellip, \|u - \thetastar\|_2 = \epsilon}
\|(u -\thetastar) - \trunc{k}(u -\thetastar)\|_2^2 \\
& \geq \epsilon^2 - \max_{u \in \ellip, \|u - \thetastar\|_2 \leq
  \epsilon} \|(u - \thetastar) - \trunc{k}(u -\thetastar)\|_2^2,
\end{align*}
where inequality (i) makes use of inequality~\eqref{EqnAdagio}.
Recalling the definition of $\trunc{k}$ from
equation~\eqref{EqnAllegro}, we find that
\begin{align*}
c_0 &~\geq \epsilon^2 - \min_{\trunc{k} \in \mathcal{P}_k}
\max_{\theta \in \ellip_{\thetastar} \cap \Ball(\epsilon)} \|\Delta -
\trunc{k}\Delta\|_2^2 \; = \; \epsilon^2 -
\kwidth_k^2(\ellip_{\thetastar} \cap \Ball(\epsilon))
\stackrel{(ii)}{\geq} \frac{\epsilon^2}{2},
\end{align*}
where inequality (ii) follows since the critical dimension $k$ was
chosen so that \mbox{$\kwidth_k(\ellip_{\thetastar} \cap
  \Ball(\epsilon)) \leq \frac{\epsilon}{\sqrt{2}}$.}

Now we are ready to bound the type II error.  For any vector $\theta$
in the alternative, we have
\begin{align*}
\Prob_\theta[\psi(y) = 0] &= \Prob_{\theta} \big[ T_k \leq \sigma^2(k
  + \sqrt{4k/\rho}) \big]
= \Prob_{\theta} \Biggr[ \frac{T_k - \sigma^2 k -
    c_0}{\sigma^2\sqrt{k}} \leq \frac{2}{\sqrt{\rho}} -
  \frac{c_0}{\sigma^2\sqrt{k}} \Biggr].
\end{align*}
Together, the two inequalities $c_0\geq \epsilon^2/2$ and $\epsilon^2
\geq 8\rho^{-1/2} \sigma^2\sqrt{k}$ imply that $c_0/\sigma^2\sqrt{k}
\geq 4\rho^{-1/2}$, whence
\begin{align*}
\Prob_\theta[\psi(y) = 0] \notag &\leq \Prob_{\theta} \left[ \frac{T_k
    - \sigma^2 k - c_0}{\sigma^2\sqrt{k}} \leq \frac{2}{\sqrt{\rho}} -
  \frac{4}{\sqrt{\rho}} \right] \leq \frac{\rho}{2},
\end{align*}
where the last inequality is again due to Chebyshev's inequality.
Therefore whenever $\epsilon^2 \geq 8\rho^{-1/2} \sigma^2\sqrt{k}$,
then type II error is upper bounded by
\begin{align}
\label{EqnTailtypeIIGen}
\sup_{\theta \in \Hyp_1} \Prob_\theta \big[\psi(y) = 0 \big] \leq
\frac{\rho}{2}.
\end{align}
Combining inequalities~\eqref{EqnTailtypeIGen}
and~\eqref{EqnTailtypeIIGen} yields the claim.


\paragraph{Proof of inequality~\eqref{EqnAdagio}:}

The only remaining detail in the proof of Theorem~\ref{ThmProjUB} is
to establish inequality~\eqref{EqnAdagio}.  It suffices to show that
the minimum
\begin{align}
  \label{EqnRabbit}
\min_{ u \in \Ellipse, \; \|u - \thetastar\|_2 \geq \epsilon}
\|\trunc{k}(\theta -\thetastar) \|^2_2
\end{align}
is always achieved by a vector $u^* \in \Ellipse$ with $\|u^* -
\thetastar\|_2 = \epsilon$.

Noting the continuity of the objective and the compactness of the
constraint set, Weierstrass' theorem ensures that the minimum is
always achieved.  Leting $\uhat$ be any vector that achieves the
minimum, we define the new vector
\begin{align*}
\ustar & \defn (1 - \alpha) \thetastar + \alpha \uhat \qquad
\mbox{where $\alpha \defn \frac{\epsilon}{\|\uhat - \thetastar\|_2} \in (0,1]$.}
\end{align*}
The convexity of $\Ellipse$ ensures that $\ustar \in \Ellipse$, and
moreover, by construction, we have $\|\ustar - \thetastar\|_2 =
\epsilon$.  Thus, the vector $\ustar$ is feasible for the original
problem~\eqref{EqnRabbit}, and by linearity of projection, we have
\begin{align*}
\|\trunc{k}(\ustar - \thetastar\|_2 & \leq (1- \alpha)
\|\trunc{k}(0)\|_2 + \alpha \|\trunc{k}(\uhat - \thetastar)\|_2 \;
\leq \; \|\trunc{k}(\uhat - \thetastar)\|_2,
\end{align*}
showing that $\ustar$ is also a minimizer, as claimed.


\subsection{Proof of Theorem~\ref{ThmLowerBound}}
\label{AppThmKLB}

We now turn to the proof of the lower bound stated in
Theorem~\ref{ThmLowerBound}.  This proof, as well as subsequent ones,
makes use of the Bernstein width as a lower bound on the Kolmogorov
width.  Accordingly, we begin by introducing some background on it, as
well stating an auxiliary lemma that shows how to obtain testing error
lower bounds via Bernstein width.


\subsubsection{Lower bounds via the Bernstein width}
\label{SecBernstein}

The Bernstein $k$-width of a compact set $C$ is the radius of the
largest $k+1$-dimensional ball that can be inscribed into $C$.
Typically, we use the Euclidean ($\ell_2$) norm to define the ball,
but more generally, we can define widths respect to the $\ell_p$-norms
with $p \in [1, \infty]$.  More precisely, for a given integer $k \geq
1$, let $\mathcal{S}_{k+1}$ denote the set of all $(k+1)$-dimensional
subspaces.  With this notation, the \emph{Bernstein $k$-width in
  $\ell_p$-norm} is given by
\begin{align}
\label{EqnBernstein}
b_{k,p}(C) & \defn \arg \max_{r \geq 0} \Big \{ \Ball^{k+1}_p(r)
\subseteq C \cap S \quad \mbox{for some subspace $S \in
  \mathcal{S}_{k+1}$} \Big \},
\end{align}
where $\Ball^{k+1}_{p}(r) \defn \{ u \in \real^{k+1} \mid \|u\|_p \leq
r\}$ denotes a $(k+1)$-dimensional $\ell_p$-ball of radius $r$
centered at zero.

An important fact is that the Kolmogorov $k$-width from
equation~\eqref{EqnKolmogorov} can always be lower bounded by the
Bernstein $k$-width in $\ell_2$---that is,
\begin{align} 
\label{EqnCompr}
\kwidth_k(C) & \geq b_{k,2}(C) \quad \mbox{for any $k = 1, 2,
  \ldots$.}
\end{align}
See Pinkus~\cite{pinkus2012n} for more details of this property.
Moreover, comparing the Bernstein widths in $\ell_2$ and $\ell_\infty$
norm, we have the sandwich relation
\begin{align}
  \label{EqnSandwichBernstein}
b_{k, \infty}(C) \; \leq \; b_{k, 2}(C) \; \leq \; \sqrt{k+1} \; b_{k,
  \infty}(C),
\end{align}
which is an elementary consequence of the fact that $1 \leq
\frac{\|u\|_2 \; \,}{\|u\|_\infty} \leq \sqrt{k+1}$ for any non-zero vector
$u \in \real^{k+1}$.

We now use these definitions to define a lower bound on the testing
error in terms of Bernstein width in $\ell_\infty$.  Define the
$\ell_\infty$-Bernstein lower critical dimension as
\begin{subequations}
\begin{align}
\label{EqnBK}
\kb(\epsilon, \thetastar, \ellip) \defn \arg \max_{1\leq k \leq \usedim}
\Big\{ k b^2_{k - 1,\infty}(\ellip_{\thetastar}) \geq \epsilon^2 \Big\},
\end{align}
and the corresponding lower critical radius
\begin{align}
\label{EqnCriBK}
  \epscritb(\thetastar ; \ellip) & \defn \sup \Big \{ \epsilon \mid
  \epsilon \leq \frac{1}{4} \sigma^2 \frac{\sqrt{\kb(\epsilon, \thetastar,
      \ellip)}}{\epsilon} \Big \}.
\end{align}
\end{subequations}
In terms of these quantities, we have the following auxiliary result:
\begin{lems}
\label{LemBernstein}
For any vector $\thetastar \in \ellip$, we have
\begin{align}
\label{EqnLBGen} 
\inf_{\psi} \UNIERR(\psi; \thetastar, \delta) \geq \frac{1}{2} \qquad
\text{for all } \delta \leq \epscritb(\thetastar ; \ellip).
\end{align}
\end{lems}
\noindent See Appendix~\ref{AppThmBLB} for the proof of this lemma.

Lemma~\ref{LemBernstein} is useful because evaluating the Bernstein
width is often easier than evaluating the Kolmogorov width.  Note that
we also have a (possibly weaker) lower bound in term of
$\ell_2$-Bernstein widths.  In particular, suppose that we define the
$\ell_2$-Bernstein critical dimension
\begin{align}
\label{EqnBernAlt}
  \newkb(\epsilon, \thetastar, \ellip) \defn \arg \max_{1\leq k \leq \usedim}
  \Big\{ b^2_{k - 1,2}(\ellip_{\thetastar}) \geq \epsilon^2 \Big\}
\end{align}  
and then define the critical radius $\epscritb$ using $\newkb$.  By
the sandwich relation~\eqref{EqnSandwichBernstein}, the lower
bound~\eqref{EqnLBGen} still holds.


\subsubsection{Main portion of proof}

Let $D \defn \max_{\theta \in \Ellipse} \|\theta - \thetastar\|_2$
denote the diameter of the ellipse with respect to $\thetastar$.  We
need only consider testing radii $\epsilon \leq D$, since otherwise,
the alternative is empty.  Moreover, we need only consider noise
standard deviations $\sigma \leq D$, since otherwise even the simple
(non-compound) testing problem is hard.

Next, we claim that $\epsopt^2(\thetastar ; \ellip) \geq \sigma^2$.
In order to establish this claim, we consider a simple hypothesis test
between a pair $(\thetastar, \theta)$ both in the ellipse, and such
that $\ltwo{\thetastar - \theta} = \delta$ for some \mbox{$\delta \in
  [\sigma, D]$.}  This single-to-single test requires separation
$\delta^2 \geq \sigma^2$, a classical fact that can be proved by
various arguments, for example, by analyzing the likelihood ratio and
using Neyman-Pearson lemma.  Since our composite testing
problem~\eqref{EqnMainTesting} is at least as hard, it follows that
$\epsopt^2(\thetastar ; \ellip) \geq \sigma^2$.

For the remainder of our analysis, we assume that $\sqrt{\lwk} \geq
18$.  (Otherwise, the squared critical radius $\epscritu$ from in
equation~\eqref{EqnRadCritL}) is upper bounded by $9\sigma^2/2$, so
that the claim of Theorem~\ref{ThmLowerBound} follows by adjusting the
constant $c$, combined with the lower bound $\epsopt^2(\thetastar ;
\ellip) \geq \sigma^2$.)  We now divide our analysis into two cases,
depending on whether or not $\enorm{\at} \leq 1/2$.

\paragraph{Case 1:} First, suppose that $\enorm{\at} \leq 1/2$, which
implies that $\Rfun(\delta) \leq (\enorm{\at}-1)^2\leq 1$.  In this
case, we apply the $\ell_2$-Bernstein-width based lower bound from
Lemma~\ref{LemBernstein} to establish our theorem.  Define the
quantities $\newkb(\epsilon, \thetastar, \ellip) \defn \arg \max
\limits_{1 \leq k \leq \usedim} \Big\{ b_{k -
  1,2}(\ellip_{\thetastar}\cap \Ball(a\epsilon)) \geq b\epsilon
\Big\}$, and
\begin{align*}   
\epscritb(\thetastar ; \ellip) & \defn \sup \Big \{ \epsilon \mid b^2
\epsilon \leq \frac{1}{4} \sigma^2 \frac{\sqrt{\newkb(\epsilon,
    \thetastar,\ellip)}}{\epsilon} \Big \}.
\end{align*}
With this notation, an equivalent statement of \autoref{LemBernstein}
is that $\inf_{\psi} \UNIERR(\psi; \thetastar, \delta) \geq
\frac{1}{2}$ for all $\delta \leq b\epscritb(\thetastar ; \ellip)$.
(In this statement, we have replaced $b_{k -
  1,2}(\ellip_{\thetastar})$ by $b_{k - 1,2}(\ellip_{\thetastar} \cap
\Ball(a \epsilon))$ for later theoretical convenience; doing so does
not change the definition as $a> b$.)

In order to prove Theorem~\ref{ThmLowerBound}, we need to show that
$\ccon \epscritl(\thetastar; \ellip)\leq b\epscritb(\thetastar ;
\ellip)$, and we claim that it suffices to show that $\lwk(\epsilon,
\thetastar, \ellip) \leq \newkb(\epsilon, \thetastar,
\ellip)$. Indeed, supposing that the latter inequality holds, the
definition of the lower critical radius in
equation~\eqref{EqnRadCritL} implies that
\begin{align*}
  \epscritl(\thetastar; \ellip) & \defn \sup \Big \{ \epsilon \mid
  \epsilon \leq \frac{1}{4}
  \sigma^2 \frac{\sqrt{\lwk(\epsilon, \thetastar, \ellip)}}{\epsilon} \Big \}.
\end{align*}
Comparing with the above definition of $\epscritb(\thetastar ;
\ellip)$ and using\footnote{ It should also be noted that the right
  hand sides of both fixed point equations are non-increasing function
  of $\epsilon$ while the left hand side is non-decreasing.}  the fact
that $b \in (0,1)$, we see that $\epscritl(\thetastar; \ellip) \leq
\epscritb(\thetastar ; \ellip)$.  Finally, since $c < b$, it follows
that $\ccon \epscritl(\thetastar; \ellip)\leq b\epscritb(\thetastar ;
\ellip)$.

Now it is only left for us to show that $\lwk(\epsilon, \thetastar,
\ellip) \leq \newkb(\epsilon, \thetastar, \ellip)$.  By definition of
the critical lower dimension $\lwk(\epsilon, \thetastar, \ellip)$ in
equation~\eqref{EqnK-lower}, we have
$\kwidth_{\lwk-1}(\ellip_{\thetastar} \cap \Ball(a\epsilon)) > 3b
\epsilon$.  So our proof for Case 1 will be complete once we establish
that $b_{\lwk,2}(\ellip_{\thetastar}\cap \Ball(a\epsilon)) \geq
b\epsilon$.  In order to do so, we use the following lemma.
\begin{lems}
\label{LemChain}
The following chain of inequalities hold
\begin{align}
\label{EqnByEndOfToday}
\min\{a\epsilon, \frac{1}{2}\sqrt{\mu_{k+1}}\} \,
\stackrel{(i)}{\leq}\, b_{k,2}(\ellip_{\thetastar}\cap
\Ball(a\epsilon)) \leq \kwidth_{k}(\ellip_{\thetastar}\cap
\Ball(a\epsilon)) \, \stackrel{(ii)}{\leq}\, \frac{3}{2}
\kwidth_{k}(\ellip\cap \Ball(a\epsilon)) \, \stackrel{(iii)}{=}\,
\frac{3}{2} \min\{a\epsilon, \sqrt{\mu_{k+1}}\}.
\end{align}
\end{lems}
\noindent See Appendix~\ref{AppLemChain} for the proof of this claim.

From Lemma~\ref{LemChain}, we have the string of inequalities
\begin{align*}
  3 b\epsilon < \kwidth_{\lwk-1}(\ellip_{\thetastar} \cap
  \Ball(a\epsilon)) \leq \frac{3}{2} \min\{a\epsilon,
  \sqrt{\mu_{\lwk}} \},
\end{align*}
and hence $b_{\lwk,2}(\ellip_{\thetastar}\cap \Ball(a\epsilon))\geq
\min\{a\epsilon, \frac{1}{2}\sqrt{\mu_{\lwk}}\} \geq b\epsilon$,
which completes the proof in Case 1.


\paragraph{Case 2: }  Otherwise, we may assume that $\enorm{\at} > 1/2$,
in which case $\Rfun(\delta/\ccon) \leq (\enorm{\at}-1)^2 < 1$, and
hence by definition of the function $\Rfun$, we have $\delta <
\ccon\ltwo{\at}/a$.

Our analysis in this case makes use of the following auxiliary result
(stated as Lemma 6 in Wei et al.~\cite{wei2017geometry}).  Suppose
that we observe a Gaussian random vector $y \sim N(\Delta, \sigma^2
I_{\usedim})$.  For a set $C$, we use $\UNIERR(\psi; \{0\}, C,
\delta)$ to denote the uniform testing error defined by the compound
testing problem (in the Gaussian sequence model) of distinguishing the
null $\Delta = 0$ from the compound alternative $\Delta \in C 
\cap
\Ball^c(\delta)$, where $\Ball^c(\delta) \defn \{v \mid \ltwo{v}\geq
\delta\}$.

\begin{lems} 
\label{LemChisquareBound}
For every set $C$ and probability measure $\qprob$ supported on $C
\cap \Ball^c(\delta)$, we have
\begin{align}
\label{EqnTestErrorLowerBound}
\inf_{\psi} \UNIERR(\psi; \{0\}, C, \delta) \geq 1 - \frac{1}{2}
\sqrt{\Exs_{\eta, \eta'} \, \exp \Big(\frac{\inprod{\eta}{\eta'}}{\sigma^2} \Big)-1},
\end{align} 
where $\Exs_{\eta, \eta'}$ denotes expectation with respect to an
i.i.d pair $\eta, \eta' \sim \qprob$.
\end{lems}

In order to apply Lemma~\ref{LemChisquareBound}, we need to specify a
suitable set $C$ and probability distribution $\qprob$.  For our
\mbox{application,} the relevant set is $C = \Ellipse_\thetastar$.  In
addition, we need to construct an appropriate distribution $\qprob$
with which to apply Lemma~\ref{LemChisquareBound}. To this end, we
introduce two more auxiliary results.  In the following statement, we
fix a constant $\epsilon > 0$, as well as pair $a,b \in (0,1)$ such
that $a > 3b$, and recall that $\lwk(\epsilon)$ denotes the lower
critical dimension from equation~\eqref{EqnK-lower}.  Also recall that
the ellipse norm $\enorm{\theta}^2 = \theta^T \NewMat \theta$ is
defined in terms of the diagonal matrix $\NewMat$ with diagonal
entries $\{\frac{1}{\mu_i}\}_{i=1}^\usedim$.

\begin{lems}
\label{LemPacking}
For any vector $\thetastar \in \ellip$ such that
$\norm{\thetastar}_2 > a \epsilon$, there exists a vector $\thetadag
\in \mathcal{E}$, a collection of $\usedim$-dimensional orthonormal
vectors $\{u_i \}_{i=1}^{\lwk}$ and an upper triangular matrix of the
form
\begin{align}
 \Mat \defn
 \begin{bmatrix}
   \label{EqnDefnMat}
1 & h_{3,2} & h_{4,2} & \cdots & h_{\lwk,2}\\ & 1 & h_{4,3} & \cdots &
h_{\lwk,3}\\ &&1 & \cdots & h_{\lwk,4} \\ &&&\ddots & \vdots \\ &&&&1
    \end{bmatrix} \in \real^{(\lwk-1) \times (\lwk -1)},
\end{align}
with ordered singular values $\nu_1(\Mat) \geq \nu_2(\Mat) \geq \cdots
\geq \nu_{\lwk-1}(\Mat) \geq 0$ such that:
\begin{enumerate}
\item[(a)] The vectors $u_1$, $\NewMat \thetadag$, and $\thetadag -
  \thetastar$ are all scalar multiples of one another.
\item[(b)] We have $\norm{\thetadag - \thetastar}_2 = a \epsilon$.
\item[(c)] For each $i \in [\lwk-1]$, the $i^{th}$ column $\Mat_{i,
  \cdot}$ of the matrix $\Mat$ has its squared Euclidean norm upper
  bounded as \mbox{$\norm{\Mat_{\cdot, i}}_2^2 \leq 
    a^2/9b^2$.}
\item[(d)] For each $i \in [\lwk - 1]$, each of the vectors $\thetadag
  \pm b \epsilon 
  \underbrace{\begin{bmatrix} u_2 & \cdots
    u_{\lwk}
  \end{bmatrix}}_{\defn U} H_{\cdot, i}$ belongs to the ellipse $\Ellipse$.
\item[(e)] For each pair of integers $(s,t) \in [\lwk-1] \times
  [\lwk-2]$, we have
  \begin{align}
    \label{EqnEignBound}
\nu_{s} \stackrel{(i)}{\leq} \frac{a}{3b}\sqrt{\frac{\lwk-1}{s}},
~~\text{ and }~~ \nu_{t+1} \stackrel{(ii)}{\geq} 1 - \frac{t}{\lwk-1}
- \sqrt{\frac{a^2 - 9b^2}{9b^2}}.
\end{align}
\end{enumerate}
\end{lems}
\noindent See Appendix~\ref{AppLemPacking} for the proof of this lemma.

Let us write the SVD of the matrix $\Mat$ from
equation~\eqref{EqnDefnMat} in the form $Q \Sigma V^T$.  Since $\Mat$
is square and full rank, the matrices $Q$ and $V$ are square and
orthogonal, and $\Sigma$ is a square matrix of singular values.  By
construction, we have $\Mat^T \Mat = V \Sigma^2 V^T$.  
Define matrix $\AMat \defn V^T(UH)^T \NewMat (UH) V$ and write its diagonal 
as $(A_1,\ldots,A_{\lwk-1}).$
We introduce index set $\Mdex(r) \defn \{i \mid A_i \leq A_{(r)}\}$
which contains those elements that lie below the $r$-th quantile of $\AMat$'s diagonals elements.
We assume that $(\lwk-1)/8$ is an integer; otherwise, the same argument can be applied by taking its floor function.  
Let $m_1 \defn (\lwk-1)/4$ and $m_2 \defn (\lwk-1)/2$.
Fixing a sparsity level $s \leq m_2-m_1+1$, let $\mathcal{S}$ be the set 
of all subsets of 
set $F \defn \{m_1,m_1+1, \ldots,m_2\} \cap \Mdex(7(\lwk-1)/8) $ of size $s$.  
Note that from definitions, the cardinality of $F$ is lower bounded by $(\lwk-1)/8.$

Given a subset $S \in \mathcal{S}$
and a sign vector $z^S = (z^S_1,\ldots,z^S_{\lwk-1}) \in
\{-1,0,1\}^{\lwk-1}$, define the vector
\begin{align}
\label{EqnPerturb}
 \theta^S \defn \thetastar + b \epsilon\frac{1}{4\sqrt{2s}} 
 X V z^S, \qquad \Delta^S \defn \theta^S - \thetastar,
\end{align}
where the matrix $X = U\Mat$ was previously defined in
\autoref{LemPacking}. We then make the following claim:
\begin{lems}
\label{LemHanaSleep}
Under the assumption of Lemma~\ref{LemPacking} and $\Rfun(\epsilon)
\leq (\enorm{\at}-1)^2$, for any subset $S \in \mathcal{S}$, there is
a sign vector $z^S$ supported on $S$ such that:
  \begin{align}
 \label{Eqncond1}
 \theta^S(z^S) &  \; \stackrel{(i)}{\in} \; \ellip, \qquad \text{and }~~ 
\ltwo{\Delta^S} \;\stackrel{(ii)}{\geq} \; \ccon\epsilon.
\end{align}
\end{lems}
\noindent
See Appendix~\ref{AppLemHanaSleep} for the proof of this lemma.

We now apply Lemma~\ref{LemChisquareBound} with $\qprob$ as the
uniform distribution on those perturbation vectors $\Delta^S =
\theta^S - \thetastar$ that are defined via
expression~\eqref{EqnPerturb}.  Note that the sign vector $z^S$ is
taken to be the one whose existence is guaranteed by
Lemma~\ref{LemHanaSleep}.  For any $S \in \mathcal{S}$, due to
inequality~\eqref{Eqncond1}, we have $\ltwo{\Delta^S} \geq \epsilon$,
and thus, with the choice $\delta \defn \ccon \epsilon$, the
distribution $\qprob$ is supported on set $\ellip_{\thetastar} \cap
\Ball^c(\delta)$

For simplicity of the notation, let us introduce the shorthand
notation $\tmpdim \defn m_2-m_1+1$; notat that this integer satisfies
the relations $\tmpdim = \lwk/4 \geq 81$ and $\lambda \defn
\frac{1}{32\sigma^2} b^2 \epsilon^2 \nu^2_{m_1}$, where $\nu_{m_1}$
denotes the $m_1$-th largest eigenvalue of matrix $\Mat$.

\begin{lems}
  \label{LemPackingLB}
Given the distribution $\qprob$ defined as above, we have
\begin{align}
\label{EqnPackingLB}
\Exs_{\eta,\eta' \sim \qprob} \Big[ \exp
  \big(\frac{\inprod{\eta}{\eta'}}{\sigma^2} \big) \Big] & \leq \exp
\left(- \left(1-\frac{1}{\sqrt{\tmpdim}}\right)^2 + \exp
\left(\frac{2+\lambda}{\sqrt{\tmpdim}-1}\right) \right).
\end{align}
\end{lems}
\noindent See Appendix~\ref{AppLemPackingLB} for the proof of this lemma.

We can now complete the proof of Case 2. If the radius satisfies
$\epsilon^2 \leq \frac{\sqrt{\lwk}}{4}\sigma^2$, as a consequence of
the eigenvalue bound~\eqref{EqnEignBound}, we find that
\begin{align*}
  \frac{\lambda}{\sqrt{t}} = \frac{\epsilon^2b^2
    \nu^2_{m_1}}{32\sigma^2\sqrt{m_2-m_1+1}} \leq
  \frac{\sqrt{k_\ell} \cdot b^2 \cdot
    \frac{4a^2}{9b^2}}{144\sqrt{m_2-m_1+1}} \leq \frac{1}{18},
\end{align*}
where the last inequality uses $a<1$ and $m_2-m_1+1 = \lwk/4.$

With control of these quantities in hand, we can bound the testing
error.  The right hand side in expression~\eqref{EqnPackingLB} can be
upper bounded as
\begin{align*}
  \exp\left(- \left(1-\frac{1}{\sqrt{\tmpdim}} \right)^2 + \exp
  \left(\frac{2}{\sqrt{\tmpdim}-1} +
  \frac{\sqrt{\tmpdim}}{\sqrt{\tmpdim} - 1}
  \frac{\lambda}{\sqrt{\tmpdim}} \right)\right) \leq \exp\left(-
  \left(1-\frac{1}{9}\right)^2 + \exp \left(\frac{1}{4} +
  \frac{9}{8}\cdot \frac{1}{18} \right) \right) < 2,
\end{align*}
where we use the fact that $\tmpdim = \lwk/4 \geq 81$.  As a
consequence of Lemma~\ref{LemChisquareBound}, we have
\begin{align*}
\inf_{\psi} \UNIERR(\psi; \{0\}, \ellip_{\thetastar}, c\epsilon) \geq
1 - \frac{1}{2} \Big[ \Exs_{\eta, \eta'} \, \exp(
  \frac{\inprod{\eta}{\eta'}}{\sigma^2})-1 \Big]^{1/2} \; > \; \frac{1}{2} \geq
\rho,
\end{align*}
which completes the proof of Case 2, and hence the theorem.



\subsection{Proof of Corollary~\ref{CorZero}}
\label{AppCorZero}

Our proof of Corollary~\ref{CorZero} is based on a particular way of sandwiching the upper and
lower critical radii, which we first introduce.


\subsubsection{Sandwiching the upper and lower critical radii}
\label{AppLemonTea}

Recall the upper and lower critical dimensions from
equations~\eqref{EqnK-upper} and~\eqref{EqnK-lower}, respectively.
Suppose that we can find pair a pair of functions $(\kestl, \kestu)$
such that the sandwich relation
\begin{align}
\label{EqnBDK}
  \kestl(\epsilon) \,\leq\, \lwk(\epsilon) \,\leq\, \upk(\epsilon)\leq
      \kestu(\epsilon), ~~\text{ for } \epsilon > 0.
\end{align}
holds.  Under an additional monotonicity property, any such pair can
be used to bound the interval $[\epscritl, \epscritu]$.

\begin{lems}
\label{LemEASY}
Consider a pair of functions $\kestu$ and $\kestl$ that are each
non-increasing in $\epsilon$, and satisfy the sandwich
relation~\eqref{EqnBDK}.  Then we have $\big[\epscritl, \; \epscritu
  \big] \subseteq \big[\epsl, \; \epsu \big]$, where
\begin{align*}
  \epsu \defn \inf \left \{ \epsilon \mid \epsilon \geq
  \frac{8}{\sqrt{\rho}} \sigma^2
  \frac{\sqrt{\kestu(\epsilon)}}{\epsilon} \right \} \quad \mbox{and}
  \quad \epsl \defn \sup \left \{\epsilon \mid \epsilon \leq
  \frac{1}{4} \sigma^2 \frac{\sqrt{\kestl(\epsilon)}}{\epsilon} \right
  \}.
\end{align*}
\end{lems}
\begin{proof}
  These bounds follow directly by replacing $\upk$ by $\kestu$
  (respectively $\lwk$ by $\kestl$) in the definition of $\epsu$
  (respectively $\epsl$).  The non-increasing nature of the functions
  $(\kestl, \kestu)$ ensure that the quantities $(\epsl, \epsu)$ are
  well-defined.
\end{proof}


\subsubsection{Proof of upper bound}

We prove the upper bound via an application of Lemma~\ref{LemEASY}
combined with Theorem~\ref{ThmProjUB}.  In particular, we construct
functions\footnote{Strictly speaking, these functions also depend on
  $\thetastar$ and the ellipse $\Ellipse$, but we omit this dependence
  so as to simplify notation.}  $(f_k^\ell, f_k^u)$ that sandwich
local Kolmogorov width above and below by
\begin{align}
 \label{EqnSandwichKol}
 f_k^\ell(\epsilon) ~\leq~ \kwidth_k(\ellip_{\thetastar} \cap \Ball(a
 \epsilon)) \leq \kwidth_k(\ellip_{\thetastar} \cap \Ball(\epsilon))
 ~\leq f_k^u(\epsilon) \qquad \mbox{for every integer $1\leq k \leq
   \usedim$.}
 \end{align} 
Here $a \in (0,1)$ is the constant involved in the statement of
Theorem~\ref{ThmLowerBound}.  It is then straightforward to verify
that the functions
\begin{align}
\label{EqnKest}
\kestu(\epsilon) \defn \arg \min_{1 \leq k \leq \usedim} \big \{
f_k^u(\epsilon) \leq \frac{\epsilon}{\sqrt{2}} \big \} ~~\text{ and
}~~ \kestl(\epsilon) \defn \arg \min_{1 \leq k \leq \usedim} \big\{
f_k^\ell(\epsilon) \leq 3b \epsilon \big\}.
\end{align}
satisfy the conditions of Lemma~\ref{LemEASY}, so that it can be
applied.

We begin by upper bounding the localized Kolmogorov
width~\eqref{EqnKolmogorov} of $\ellip \cap \Ball(\epsilon)$. Let
$\trunc{k}$ denote the orthogonal projection onto the first $k$
coordinates of the ellipse.  We then have
\begin{align*}
\kwidth^2_{k}(\ellip_\thetastar \cap \Ball(\epsilon)) \leq
\max_{\theta \in \ellip, \|\theta - \thetastar\|_2 \leq \epsilon}
\|(\theta - \thetastar) - \trunc{k} (\theta - \thetastar) \|^2_2
= \max_{\theta \in \ellip, \|\theta\|_2 \leq
  \epsilon}\sum_{i=k+1}^\usedim \theta_i^2.
\end{align*}
Since the scalars $\{\mu_i\}_{i=1}^\usedim$ are ranked in
non-increasing order, for any vector $\theta \in \Ellipse$, we have
\begin{align*}
\sum_{i=k+1}^\usedim \theta_i^2 = \sum_{i=k+1}^\usedim \mu_i
\frac{\theta_i^2}{\mu_i} & \; \leq \; \mu_{k+1} \sum_{i=k+1}^\usedim
\frac{\theta_i^2}{\mu_i} \; \leq \; \mu_{k+1},
\end{align*}
where the final inequality follows from the inclusion $\theta \in
\Ellipse$.  Consequently, we have proved that
\begin{align*}
\kwidth_{k}(\ellip_{\thetastar} \cap \Ball(\epsilon)) & \leq
f_k^u(\epsilon) \; \defn \; \underbrace{\min \{ \epsilon, \sqrt{\mu_{k+1}}}_{ = : f_k^u(\epsilon)}\}.
\end{align*}
It is easy to check that the function $\epsilon \mapsto
\kestu(\epsilon)$ from equation~\eqref{EqnKest} is non-increasing, so
that an application of Lemma~\ref{LemEASY} yields the claimed upper
bound~\eqref{EqnMinuet}.


\subsubsection{Proof of lower bound}

We prove the lower bound by a combination of Lemma~\ref{LemEASY} and
Theorem~\ref{ThmLowerBound}, and in order to do so, we need to lower
bound the localized Kolmogorov width.  As previously noted, the
Bernstein width~\eqref{EqnBernstein} is always a lower bound on the
Kolmogorov width, so that
\begin{align*}
  \kwidth_k(\ellip \cap \Ball(a \epsilon))\geq b_{k,2}(\ellip \cap
  \Ball(a \epsilon)) = \underbrace{\min\{a \epsilon,
    \sqrt{\mu_{k+1}}\}}_{ = : f_k^\ell(\epsilon)}
\end{align*}
where the last inequality follows from the fact that we can always
inscribe a Euclidean ball of radius $\sqrt{\mu_{k+1}}$ centered at
zero inside the ellipse $\ellip$ truncated to its first
$(k+1)$-coordinates.  It is straightforward to verify that the
corresponding $\kestl$ that is defined in \eqref{EqnKest} is
non-increasing in $\epsilon$, so that the the claimed lower
bound~\eqref{EqnMinuet} follows by applying Lemma~\ref{LemEASY} with
Theorem~\ref{ThmLowerBound}.


\subsection{Proof of Corollary~\ref{CorCorner}}
\label{AppCorCorner}

We again separate our proof into two parts, corresponding to the upper
bound and lower bounds.  The upper bound is proved via an application
of Lemma~\ref{LemEASY} and Theorem~\ref{ThmProjUB}.  On the other
hand, we prove the lower bound by using Theorem~\ref{ThmLowerBound}
and Lemma~\ref{LemBernstein} in conjunction with Lemma~\ref{LemEASY}.

\subsubsection{Proof of upper bound}

Let us introduce the convenient shorthand $\at_{-s} \defn (\at_1,
\ldots, \at_{s-1}, \at_{s+1}, \ldots, \at_\usedim)$. In terms of the
the local minimax testing radius $\epsopt(\thetastar; \ellip)$
previously defined in equation~\eqref{EqnDefnLocalMinimaxOPT},
Theorem~\ref{ThmProjUB} guarantees that $\epsopt(\thetastar) \leq
\epscritu(\thetastar; \ellip)$.  Consequently, in order to establish
the upper bound stated in Corollary~\ref{CorCorner}, it suffices to
show that for any vector $\at$ with $\at_s \geq \sqrt{\mu_s} -
\newepscritu(s,\ellip)$ and $\at_{-s} = 0$, we have
\mbox{$\epscritu(\thetastar ; \ellip) \leq \newepscritu(s, \ellip)$.}
In order to do so, it suffices to show that
\begin{align}
\label{EqnSushi}
  \kwidth_{\mind_u(\delta, s)}(\ellip_{\thetastar} \cap \Ball(\delta))
  \leq \frac{\delta}{\sqrt{2}} \qquad \mbox{for each $\delta > 0$.}
\end{align}
It then follows that $\mind_u(\delta; s) \geq \upk(\delta)$, from
which an application of Lemma~\ref{LemEASY} implies that
\mbox{$\epscritu(s, \ellip) \leq \newepscritu(s, \ellip)$,} as
desired.

Accordingly, the remainder of our argument is devoted to establishing
inequality~\eqref{EqnSushi}.  Let $\trunc{\mind}$ denote the
orthogonal projection onto the span of the first $\mind$ standard
basis vectors $\{e_i\}_{i=1}^\mind$.  By
definition~\eqref{EqnKolmogorov}, the localized Kolmogorov width is
upper bounded by
\begin{align*}
\kwidth_{\mind}(\ellip_{\thetastar} \cap \Ball(\delta)) \leq
\underbrace{\max_{\theta \in C} \sqrt{ \sum_{i=\mind+1}^d
    \theta_i^2}}_{= \,: T_{\mind}^u} \qquad \mbox{where} \quad C \defn
\big \{ \theta \in \real^\usedim \mid \sum_{i=1}^\usedim
\frac{\theta_i^2}{\mu_i} \leq 1, ~~ (\theta_s - \at_s)^2 + \sum_{i\neq
  s} \theta_i^2 \leq \delta^2 \big \},
\end{align*}
where we have used the fact that $\thetastar_{-s} = 0.$

Let us introduce the shorthand $\mindnots \defn
\{1,2,\ldots,s-1,s+1,\ldots,\mind\}$.  Observe that setting $\theta_j
\neq 0$ for any $j \in \mindnots$ only makes the constraints defining
$C$ more difficult to satisfy, and does not improve the objective
function defining $T_\mind^u$.  Therefore, we can reduce the problem to
\begin{align*}
  \max_{\theta_s, \{\theta_i\}_{i=\mind+1}^\usedim}
  \sum_{i=\mind+1}^\usedim \theta_i^2, \quad \mbox{such that} \quad
  \sum_{i=\mind+1}^\usedim \frac{\theta^2_i}{\mu_i} \leq 1 \text{ and
  } (\theta_s - \at_s)^2 + \sum_{i=\mind+1}^{\usedim} \theta^2_i \leq
  \delta^2.
\end{align*}
Since the scalars $\{\mu_i\}_{i=\mind+1}^\usedim$ are ranked in
non-increasing order, any optimal solution must satisfy $\theta_i = 0$
for all $i \in \{\mind+2, \ldots, \usedim\}$. Putting pieces together,
the optimal solution is defined by a pair $(\theta_s,
\theta_{\mind+1})$ that satisfy the relations
\begin{align*}
(\theta_s - \at_s)^2 + \theta_{\mind+1}^2 &= \delta^2 ~~\text{ and }~~
  \frac{\theta^2_s}{\mu_s} + \frac{\theta^2_{\mind+1}}{\mu_{\mind+1}}
  = 1.
\end{align*}
Solving these equations yields
\begin{align} \label{EqnRDUair}
  \theta_s = \frac{\at_s}{1-t} - \sqrt{ \frac{\delta^2 -
      \mu_{\mind+1}}{1-t} + \frac{t(\at_s)^2}{(1-t)^2}},
 \quad \mbox{and} \quad
  \theta_{\mind+1}^2 = \delta^2 - \left( \sqrt{ \frac{\delta^2 -
      \mu_{\mind+1}}{1-t} + \frac{t(\at_s)^2}{(1-t)^2}} -
  \frac{t\at_s}{1-t} \right)^2,
\end{align}
where we define $t \defn \mu_{\mind+1}/\mu_s.$ We introduce the
following lemma.
\begin{lems}
\label{LemCalculation}
For any $\delta \geq \newepscritu(s,\ellip)$ and $\sqrt{\mu_s} -
\newepscritu(s,\ellip) \leq \at_s \leq \sqrt{\mu_s}$, we have
$\theta_{\mind+1}^2 \leq \frac{\delta^2}{2}$.
\end{lems}
\noindent See Appendix~\ref{AppCalc} for the proof of this result. 

\noindent Therefore, the localized Kolmogorov width is at most
$\kwidth^2_{\mind}(\ellip_{\thetastar} \cap \Ball(\delta)) \leq
(T_{\mind}^u)^2 = \theta_{\mind+1}^2 \leq \frac{\delta^2}{2}$, which
completes the proof of the upper bound stated in
Corollary~\ref{CorCorner}.


\subsubsection{Proof of lower bound}  
In terms of the the local minimax testing radius $\epsopt(\thetastar;
\ellip)$ previously defined in
equation~\eqref{EqnDefnLocalMinimaxOPT}, Lemma~\ref{LemBernstein}
guarantees that $\epsopt(\thetastar) \geq \epscritb(\thetastar;
\ellip)$.  Therefore, in order to establish the lower bound stated in
Corollary~\ref{CorCorner}, it is sufficient for us to show that for
any vector $\at$ with $\at_s \leq \sqrt{\mu_s} -
\newepscritl(s,\ellip)$ and $\at_{-s} = 0$, we have
$\epscritb(\thetastar ; \ellip) \geq \newepscritl(s, \ellip)$.
Suppose that we can show that
\begin{align}
\label{EqnDuck}
  \mind b^2_{\mind - 1,\infty}(\ellip_{\thetastar}) \geq \delta^2,
\end{align}
where we have introduced the shorthand $\mind \defn \newlwk(\delta;
s)$.  It then follows that then $\kb(\delta) \geq \mind$, from which
an application of Lemma~\ref{LemEASY} guarantees that
$\epscritb(\thetastar ; \ellip) \geq \newepscritl(s, \ellip)$.
Here we apply Lemma~\ref{LemEASY} with respect to $\kb$ and $\epscritb$
replacing quantities $\lwk$ and $\epscritl$ in the statement 
which holds by the same argument of the original proof.

The remainder of our argument is to devoted to proving
inequality~\eqref{EqnDuck}.  Consider the set
\begin{align*}
\mathcal{H} & \defn \left \{\theta \mid \theta_s = \at_s, ~
\theta_{m+1}, \ldots, \theta_d = 0, ~
\theta_i \in [-\frac{\delta}{\sqrt{|\mathcal{M}|}}, \frac{
    \delta}{\sqrt{|\mathcal{M}|}}],~ i \in \mathcal{M} \right\},
~~\mathcal{M} \defn \{1\leq i \leq \mind, i\neq s\}.
\end{align*}
We claim that the set $\mathcal{H}$ is contained within the ellipse
$\ellip$.  In order to see this, first note that $\mu_i$ are ranked in
a non-increasing order, so we have $\frac{\delta^2}{|\mathcal{M}|}
\sum_{i\neq s}^{m} \frac{1}{\mu_i} \leq \delta^2/\mu_{m}$.  It implies
\begin{align} \label{EqnEllCornerLB}
\enorm{\theta}^2 = \frac{(\at_s)^2}{\mu_s} + \sum_{i\neq s}^{m}
\frac{\delta^2}{\mu_i|\mathcal{M}|}
\leq \frac{(\at_s)^2}{\mu_s} + \frac{\delta^2}{\mu_{m}}
= 1 - \frac{2 w}{\sqrt{\mu_s}} + \frac{w^2}{\mu_s} +
\frac{\delta^2}{\mu_{m}},
\end{align}
where we plugged in $\at_s = \sqrt{\mu_s} - w$ with $w \geq
\newepscritl(s,\ellip).$ On one hand, inequality $w \leq \sqrt{\mu_s}$
guarantees that $\frac{w^2}{\mu_s} - \frac{w}{\sqrt{\mu_s}}<0.$ On the
other hand, inequalities $\mu_k^2 \geq \delta^2\mu_s$ and $\delta\leq
\newepscritl(s,\ellip)$ yield that $\frac{\delta^2}{\mu_{k}} \leq
\frac{\delta}{\sqrt{\mu_s}} \leq \frac{w}{\sqrt{\mu_s}}$.  Combining
these two facts, we have $\enorm{\theta}^2 \leq 1$ for every $\theta
\in \mathcal{H}$, namely set $\mathcal{H} \subset \ellip$.

It means that we are able to inscribe a $|\mathcal{M}|-1$-dimensional
$\ell_{\infty}$ ball with radius $\delta/\sqrt{|\mathcal{M}|}$ into
the ellipse $\ellip$, namely $\mind b^2_{\mind -
  1,\infty}(\ellip_{\thetastar}) \geq \delta^2$.  Putting pieces
together, we finish the proof of the lower bound of
Corollary~\ref{CorCorner}.


\section{Discussion}
\label{SecDiscussion}

In this paper, we have studied the local geometry of compound testing
problems in ellipses.  Our main contribution was to show that it is
possible to obtain a sharp characterization of the localized minimax
testing radius at a given vector $\thetastar$ in terms of an equation
involving the Kolmogorov width of a particular set.  This set involves
the geometry of the ellipse, but is localized to the vector
$\thetastar$ via intersection with a Euclidean ball.  This form of
localization and the resulting fixed point equation is similar to
localization techniques that are used in proving optimal rates for
estimation (e.g.,~\cite{vandeGeer}), and our work demonstrates that
similar concepts are also useful in establishing optimal testing
rates.  We anticipate that this perspective will be useful in studying
adaptivity and local geometry in other testing problems.


\subsection*{Acknowledgments}
This work was partially supported by Office of Naval Research grant
DOD ONR-N00014, and National Science Foundation grant NSF-DMS-1612948
to MJW.

\vspace*{2cm}



\appendix
\section{Proof of Lemma~\ref{LemBernstein}}
\label{AppThmBLB}

In order to simplify notation, we use $k$ as a shorthand for $\kb$
throughout this proof.  Our goal is to prove that if $\epsilon^2 \leq
\frac{\sqrt{k}\sigma^2}{4}$, then any test has testing error at least
$1/2$.  Let $e_i \in \real^\usedim$ denote the standard basis vector
with the $i$-th coordinate equal to one. The definition of $k$
guarantees that we can construct a $2^k$-sized collection of
perturbation vectors of the form
\begin{align}
  \label{EqnPerturbTilde}
\theta_b = \thetastar + \frac{\epsilon}{\sqrt{k}} \sum_{i=1}^k b_i
e_i, ~~~b \in \{-1, +1\}^k
\end{align}
such that each $\theta_b$ lies on the boundary of the $\ell_{\infty}$
ball around $\thetastar$ with radius $\frac{\epsilon}{\sqrt{k}}$.
This last property ensures that each $\theta_b$ belongs to the ellipse
$\ellip$.  At the same time, we have \mbox{$\|\theta_b -
  \thetastar\|^2 = k(\frac{\epsilon}{\sqrt{k}})^2 = \epsilon^2$,}
valid for all Boolean vectors $b \in \{-1, +1\}^k$, so that each
perturbation vector $\theta_b$ indexes a distribution in the
alternative hypothesis class $\Hyp_1$.

By Lemma~\ref{LemChisquareBound}, we can lower bound the testing error
as
\begin{align*}
  \inf_{\psi} \UNIERR(\psi; \{\thetastar\}, \ellip, \epsilon) =
  \inf_{\psi} \UNIERR(\psi; \{0\}, \ellip_{\thetastar}, \epsilon) \geq
  1 - \frac{1}{2} \sqrt{\Exs_{\eta, \eta'\sim \qprob} \,
    \exp(\frac{\inprod{\eta}{\eta'}}{\sigma^2} )-1}.
\end{align*}

Let $\qprob$ denote the uniform distribution over the set $\{\theta_b
- \thetastar, b \in \{-1, +1\}^k \}$.  Introducing the shorthand
$\numvec \defn 2^{k}$, we have
\begin{align}
  \label{EqnChisquare}
  \Exs_{\eta, \eta'\sim \qprob} \,
  \exp(\frac{\inprod{\eta}{\eta'}}{\sigma^2} ) \; = \;
  \frac{1}{\numvec^2} \sum_{b,b'}\exp(\frac{(\theta_b - \thetastar)^T
    (\theta_{b'} - \thetastar)}{\sigma^2}) & = \frac{1}{\numvec}
  \sum_{b}\exp(\frac{(\theta_b -
    \thetastar)^T(\epsilon\text{1}_d/\sqrt{k})}{\sigma^2})\\
\notag &= 2^{-k} \sum_{i=0}^k \binom{k}{i}
\exp(\frac{\epsilon^2(k-2i)}{k \sigma^2})\\
\notag &= \Big( \frac{\exp(\epsilon^2/k\sigma^2) +
  \exp(-\epsilon^2/k\sigma^2)}{2} \Big)^k,
\end{align}
where we have used the symmetry of the problem.  It can be verified
via elementary calculation that $\frac{\exp(x) + \exp(-x)}{2} \leq
1+x^2$ for all $|x| \leq 1/2$.  Setting $x =
\frac{\epsilon^2}{k\sigma^2}$, we have $x \leq \frac{1}{4\sqrt{k}} <
1/2$, and hence
\begin{align*}
  \frac{1}{\numvec^2} \sum_{b,b'}\exp(\frac{(\theta_b - \thetastar)^T
    (\theta_{b'} - \thetastar)}{\sigma^2}) \leq (1+
  \frac{\epsilon^4}{k^2 \sigma^4})^k \leq
  \exp(\frac{\epsilon^4}{k\sigma^4}),
\end{align*}
where in the last step we used the standard bound $1 + x \leq e^x$.
Therefore, the testing error is lower bounded by $1 - \frac{1}{2}
\sqrt{e^{\epsilon^4/k\sigma^4}-1}$.  Thus, whenever \mbox{$\epsilon^2
  \leq \sqrt{k}\sigma^2/4$,} the testing error is lower bounded as
$\inf_{\psi} \UNIERR(\psi; \{\thetastar\}, \ellip, \delta) \geq 1 -
\frac{1}{2} \sqrt{e^{1/16} - 1} \geq \frac{1}{2}$, which completes the
proof of Lemma~\ref{LemBernstein}.


\section{Proofs for Theorem~\ref{ThmLowerBound}}

In this appendix, we collect the proofs of various auxiliary results involved
in the proof of Theorem~\ref{ThmLowerBound}.

\subsection{Proof of Lemma~\ref{LemChain}}
\label{AppLemChain}
Note that the inequality $b_{k,2}(\ellip_{\thetastar}\cap
\Ball(a\epsilon)) \leq \kwidth_{k}(\ellip_{\thetastar}\cap
\Ball(a\epsilon))$ follows as an immediate consequence of the relation
between widths.  It remains to prove inequalities (i)-(iii).

\paragraph*{Proof of inequalities (i) and (iii):}

Recalling the definition of the Bernstein width, we claim that
$\kwidth_k((\ellip\cap \Ball(a\epsilon))) = b_{k,2}(\ellip\cap
\Ball(a\epsilon)) = \min\{a \epsilon, \sqrt{\mu_{k+1}}\}$.  In order
to prove this claim, it suffices to show that
\begin{align*}
\kwidth_k(\ellip \cap \Ball(a\epsilon)) \stackrel{(i)}{\leq} \min\{a
\epsilon, \sqrt{\mu_{k+1}}\}, \quad \mbox{and} \quad b_{k,2}(\ellip
\cap \Ball(a\epsilon)) \stackrel{(ii)}{\geq} \min\{a \epsilon,
\sqrt{\mu_{k+1}}\}.
\end{align*}

Inequality (i) follows easily by direct calculation after taking the
$k$-dimensional projection in the definition of Kolmogorov width to be
projecting to the span of $\{e_1,\ldots,e_k\}$.  To show the second
part, note that any vector $v \in \real^\usedim$ with $\ltwo{v} \leq
\sqrt{\mu_{k+1}}$ and $v_i = 0$ for all $i = k+2, \ldots, \usedim$
satisfies
\begin{align*}
\enorm{v}^2 = \sum_{i=1}^\usedim \frac{v_i^2}{\mu_i} =
\sum_{i=1}^{k+1} \frac{v_i^2}{\mu_i} \, \stackrel{(a)}{\leq} \,
\frac{1}{\mu_{k+1}} \sum_{i=1}^{k+1} v_i^2 
 \stackrel{(b)}{\leq} 1,
\end{align*}
where inequality (a) follows from the non-increasing ordering of the
sequence $\{\mu_i \}_{i=1}^\usedim$ and inequality (b) follows from
the structure of vector $v.$ Consequently, the ellipse $\ellip$
contains a $k+1$-dimensional $\ell_2$ ball centered at $0$ of radius
$\sqrt{\mu_{k+1}}$.  Therefore we have $b_{k,2}(\ellip\cap
\Ball(a\epsilon)) \geq \min\{a \epsilon, \sqrt{\mu_{k+1}}\}$. We
complete the proof of inequality (iii).

Consider every vector $v \in \ellip$ that lies in the
$k+1$-dimensional $\ell_2$ ball that specified above.  Since
$\enorm{\at} \leq 1/2$, we have $\enorm{2\at} \leq 1$ which implies
that $2\at \in \ellip.$ By the convexity of $\ellip$, we have $\at +
v/2 \in \ellip$.  It means that $b_{k,2}(\ellip_{\thetastar} \cap
\Ball(a\epsilon)) \geq \min\{a \epsilon,
\frac{1}{2}\sqrt{\mu_{k+1}}\}$, which completes the proof of
inequality (i).


\paragraph*{Proof of inequality (ii):} 

Again consider projecting to the span of $\{e_1,\ldots,e_k\}$,
therefore the Kolmogorov width can be controlled as
\begin{align*}
\kwidth_{k}^2(\ellip_{\thetastar}\cap \Ball(a\epsilon)) \leq
(a\epsilon)^2 \wedge \max \Big\{\sum_{i=k+1}^{\usedim} \Delta_i^2
~\mid~ \enorm{\at + \Delta} \leq 1, ~\ltwo{\Delta} \leq a \epsilon
\Big\}.
\end{align*}
By the triangle inequality, we obtain $\enorm{\Delta} \leq 1+
\enorm{\at} \leq 3/2$ where the last inequality uses $\enorm{\at} \leq
1/2.$ Letting $\Delta_{k+1} = \sqrt{\mu_{k+1}}$ and $\Delta_i = 0$ for
$i\neq k+1$, we have $\kwidth_{k}(\ellip_{\thetastar}\cap
\Ball(a\epsilon)) \leq \min\{a\epsilon, 3/2\sqrt{\mu_{k+1}}\}$ which
thus proves the inequality (ii).


\subsection{Proof of \autoref{LemPacking}}
\label{AppLemPacking}

We break the proof of Lemma~\ref{LemPacking} into two parts. In the
first part, we construct the vector $\thetadag$ and the collection
$\{u_i\}_{i=1}^{k_\ell}$, and show that the properties (a)--(d) hold.
In the second part, we show that matrix $\Mat$ satisfies the
eigenvalue property.

\subsubsection{Part I}

Recall that the ellipse norm is defined as $\enorm{\theta}^2 =
\sum_{i=1}^{\usedim} \frac{\theta_i^2}{\mu_i}$, so that $\theta \in
\ellip$ is equivalent to $\enorm{\theta} \leq 1$. Recall that
$\NewMat$ denotes a diagonal matrix with diagonal entries
$(1/\mu_1,\ldots,1/\mu_\usedim)$.

\paragraph{Constructing $\thetadag$:}

Let us first define a vector $\thetadag \in \ellip$ that satisfies the
assumptions in Lemma~\ref{LemPacking}.  Define the function $\psi:(0,
\infty) \rightarrow (0, \infty)$ via
\begin{align*}
  \psi(r) & \defn r^2 \ltwo{\NewMat (\Ind_\usedim + r\NewMat)^{-1}
    \thetastar}^2 = r^2 \sum_{i=1}^d \frac{1}{(1 +
    r/\mu_i)^2}\frac{(\thetastar_i)^2}{\mu_i^2} = \sum_{i=1}^d
  \frac{r^2}{(\mu_i + r)^2} (\thetastar_i)^2.
\end{align*}
Note that $\psi$ is a continuous and non-decreasing function on
$[0,\infty)$ such that $\psi(0) = 0$ and $\lim_{r \rightarrow +\infty}
  \psi(r) = \norm{\thetastar}$.  Since $a \in (0,1)$, there must
  exists some $r > 0$ such that $\psi(r) = a \epsilon$. Given this
  choice of $r$, we then define $\thetadag \defn (\Ind_\usedim
  +r\NewMat)^{-1} \thetastar$.  Since $r > 0$, our definition ensures
  that $\thetadag \in \ellip$: more precisely, we have
\begin{align*}
\enorm{\thetadag}^2 = \sum_{i=1}^d \frac{1}{(1+r/\mu_i)^2}
\frac{(\thetastar_i)^2}{\mu_i} < \enorm{\thetastar}^2 \leq 1.
\end{align*}

\vspace*{0.5cm}

Now we are ready to construct orthogonal unit vectors
$u_1,\ldots,u_{\lwk}$ such that proper perturbations of $\thetadag$
towards linear combinations of those directions still lie in the set
$\ellip.$ We do this in a sequential way.

\paragraph{Constructing the vector $u_1$.}

Defining the vector $v_1 \defn \thetastar - \thetadag$, the definition
of $\thetadag$ guarantees that $v_1$ is parallel to both $\thetadag -
\thetastar$ and $\NewMat \thetadag$, so that condition 1 of the lemma
holds.  Note also that $\enorm{\thetadag - v_1}^2 = \enorm{\thetadag +
  v_1}^2 - 4 (\thetadag)^T \NewMat v_1 \leq 1$ because $\thetadag +
v_1 = \thetastar \in \ellip$ and $ (\thetadag)^T \NewMat v_1 = r
\norm{\NewMat\thetadag}_2^2 \geq 0$.  These properties guarantee the
inclusion $\thetadag \pm v_1 \in \ellip$, whence by convexity, the
ellipse $\ellip$ contains the line segment connecting the two points
$\thetadag \pm v_1$. We let $u_1$ be the normalized version of
$v_1$---namely, $u_1 \defn v_1 / \ltwo{v_1}$.


\paragraph{Constructing $u_2$.}

Without loss of generality, we can assume that $\lwk \geq 2$.  By
definition of $\lwk$, there exists a vector $\Delta \in
\ellip_{\thetastar} \cap \Ball(a\epsilon)$ satisfying $\norm{\Delta -
  \Pi_{v_1}(\Delta)}_2 \geq 3b \epsilon$.  Accordingly, we may define
\begin{align}
  \label{equation:delta_def}
\Delta_2 & \defn \arg \max_{\substack{ \Delta \in \ellip_{\thetastar}
    \cap \Ball(a\epsilon) : \\ \norm{\Delta - \Pi_{v_1}(\Delta)}_2
    \geq 3 b \epsilon }} \enorm{\Delta - \Pi_{v_1}(\Delta)}.
\end{align}
With this choice, we set $v_2 \defn \Delta_2 - \Pi_{v_1}(\Delta_2)$
and $u_2 \defn v_2/\norm{v_2}_2$.  Note that these choices ensure that
$\ltwo{v_2} \geq 3 b \epsilon$ and $v_2 \perp v_1$.

In order to complete the proof, it suffices to show that $\thetadag
\pm \frac{1}{3} v_2 \in \ellip$. Indeed, when this inclusion holds,
then we have $\thetadag \pm b\epsilon \begin{bmatrix} u_2 & \cdots
  u_{\lwk}
  \end{bmatrix} H_{\cdot, 1} = \thetadag \pm
b\epsilon u_2 \in \ellip.$ In order to show that the required
inclusion holds, we begin by noting that since $\ellip$ is a convex
set containing $\thetadag - v_1$ and $\thetastar + \Delta_2$, it
contains the segment connecting these two points.  In particular, we
have
\begin{align*}
\thetadag + \frac{\ltwo{(\thetadag - v_1) -
    \thetadag}}{\ltwo{(\thetadag - v_1) - (\thetastar +
    \Pi_{v_1}(\Delta_2))}} v_2 = \thetadag +
\frac{\ltwo{v_1}}{\ltwo{-2 v_1- \Pi_{v_1}(\Delta_2)}} v_2 \in
\Ellipse,
\end{align*}
a result that follows from the proportional segments
theorem~\cite{friedrich2008elementary}, when we consider a line
passing through $\thetadag$ that is parallel to $v_2$. See also
Figure~\ref{fig:geometry} for an illustration.
\begin{figure}[H]
\begin{center}
  \widgraph{0.6\textwidth}{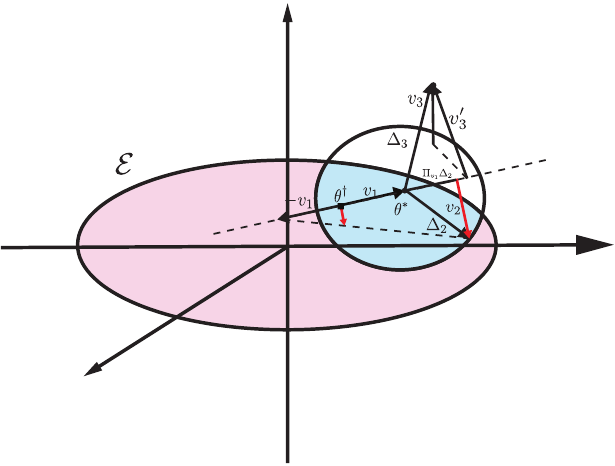}
  \caption{A geometric illustration of the proof.}
  \label{fig:geometry}
\end{center}
\end{figure}

Since $\ltwo{\Delta_2}
\leq a\epsilon \leq \ltwo{v_1}$ and $\ltwo{\Pi_{v_1}(\Delta_2)} \leq
\ltwo{\Delta_2}$, we thus have $\ltwo{-2v_1 - \Pi_{v_1}(\Delta_2)}
\leq 3 \ltwo{v_1}$ and it implies that $\thetadag + \frac{1}{3} v_2
\in \ellip$.  
Finally, using the fact that $\thetadag + \frac{1}{3} v_2 \in \ellip$
and $v_2^T \NewMat \thetadag = 0$, we have
\begin{align*}
  \enorm{\thetadag - \frac{1}{3} v_2}^2 = \enorm{\thetadag +
    \frac{1}{3} v_2}^2 - \frac{2}{3} v_2^T \NewMat \thetadag \leq 1,
\end{align*}
which completes the proof.


\paragraph{Constructing the remaining $u_i$.}

We sequentially construct pairs $(v_i, v'_i)$ for $i=3,\ldots,\lwk$ as
follows and $u_i$ is a just a scaled version of $v_i$.
Given $v_1,\ldots,v_{i-1}$ for $i \leq \lwk$, by definition of $\lwk$, 
there still exists
some $\Delta_i \in \ellip_{\thetastar} \cap \Ball(a\epsilon)$ such
that $\norm{\Delta_i - \Pi_{v_1,\ldots,v_{i-1}}(\Delta_i)}_2 \geq 3b
\epsilon$.  We then define
\begin{align*}
  v_i \defn \Delta_i - \Pi_{v_1,\ldots,v_{i-1}}(\Delta_i), \quad
  \mbox{and} \quad v'_i \defn \Delta_i - \Pi_{v_1}(\Delta_i),
\end{align*}
which ensures that $\ltwo{v_i} \geq 3b\epsilon$ and
both $v_i$ and $v_i'$ are perpendicular to $v_1.$
Finally, we set \mbox{$u_i \defn v_i/\norm{v_i}_2$}.

We now claim that
\begin{align}
\label{EqnStepIII}
\thetadag \pm \frac{1}{3} v_i' \in \ellip.
\end{align}
Taking this claim as given for the moment, let us first 
establish property (c) and (d). 
Observe that the vectors
vectors $v_1,\ldots,v_i$ are orthogonal, so we can write $v'_i$ as the
linear combination of $v_2,\ldots,v_i$.  For $i \geq 3$, there exist
scalars $r_{i,j}$ such that
\begin{align*}
v'_i = h_{i,2} v_2 + \cdots + h_{i,i-1} v_{i-1} + v_i,
\end{align*}
and by orthogonality of those $v_i$, we have 
\begin{align*}
\norm{v'_i - v_i}_2^2 = h_{i,1}^2 \ltwo{v_2}^2 + \cdots + h_{i,i-1}^2
\ltwo{v_{i-1}}^2.
\end{align*}
For $i \geq 3$, since $(v'_i - v_i) \perp v_i$, it is further guaranteed that 
\begin{align*}
\norm{v'_i - v_i}_2^2 & = \norm{v'_i}_2^2 - \norm{v_i}_2^2 \leq (a^2 -
9b^2) \epsilon^2
\end{align*}
where the last inequality follows from fact $v'_i \in
\Ball(a\epsilon)$ and the definition of fact that $\ltwo{v_i} \geq 3 b
\epsilon$.  Putting the pieces together, we have
\begin{align*}
h_{i,1}^2 + \cdots + h_{i,i-1}^2 \leq
\frac{(a^2-9b^2)\epsilon^2}{9b^2\epsilon^2} = \frac{a^2-9b^2}{9b^2}.
\end{align*}
This proves property property (c) where $\norm{\Mat_{i,\cdot}}_2^2 \leq 1 +
(a^2-9b^2)/9b^2$.  Thus by relation~\eqref{EqnStepIII}, for $i \geq 3$
we have shown
\begin{align*}
\thetadag \pm b\epsilon x_{i-1} = \thetadag \pm b\epsilon(h_{i,2} u_2
+ \cdots + h_{i,i-1} u_{i-1} + u_i) \in \ellip,
\end{align*}
which completes the proof of property (d).

It is only left to prove relation~\eqref{EqnStepIII}.  To see this
fact, first note $\norm{v_i}_2 \leq \norm{v'_i}_2$ due to Pythagoras
theorem.  Also, by the maximality of $v_2$ in inequality
\eqref{equation:delta_def}, note that $\enorm{v'_i} \leq
\enorm{v_2}$. Thus,
\begin{align*}
\enorm{\thetadag + \frac{1}{3} v'_i}^2 = \enorm{\thetadag}^2 +
\frac{1}{9} \enorm{v'_i}^2 \leq \enorm{\thetadag}^2 + \frac{1}{9}
\enorm{v_{2}}^2 = \enorm{\thetadag + \frac{1}{3} v_2}^2 \leq 1,
\end{align*}
where we have used the fact that $v'_i$ and $v_2$ are perpendicular to
$v_1 = r\NewMat\thetadag$ to ignore the cross terms, as well as the
fact that $\thetadag + \frac{1}{3} v_{2} \in \ellip$.  Thus $\thetadag
+ \frac{1}{3} v'_i \in \ellip$. We can also show $\thetadag -
\frac{1}{3} v'_i \in \ellip$ similarly as before
since $\thetadag \NewMat v_i' = 0.$


\subsubsection{Part II}

We deal with inequalities (i) and (ii) separately.  
For any $s \in [\lwk - 1]$, we have 
\begin{align*}
s \nu_{s}^2 \leq 
\sum_{i=1}^{\lwk - 1} \nu_i^2 = \trace(\Mat^T \Mat) \leq \sum_{i=1}^{\lwk - 1}
\ltwo{\Mat_{i,\cdot}}^2 \leq (\lwk-1)\frac{a^2}{9b^2},
\end{align*}
where the last inequality uses property (c).
It implies that $s \nu_{s}^2 \leq (\lwk-1)\frac{a^2}{9b^2}$ which proves inequality (i).

For any $t \in [\lwk -2]$, 
we claim that for any $t$-dimensional subspace $W \subseteq
\real^m$, there exists some $\ell$, such that standard basis vector $e_\ell$ satisfies that 
$\norm{e_\ell - \Pi_W(e_\ell)}_2^2 \geq 1 - \frac{t}{m}$.  In order
to prove this claim, we take an orthonormal basis $z_1,\ldots,z_ t$ of
$W$ and extend it to an orthonormal basis $z_1,\ldots,z_m$ for
$\real^m$. We then have
\begin{align*}
\sum_{i=1}^m \ltwo{e_i - \Pi_W(e_i)}^2 = \sum_{i=1}^m \sum_{j=t+1}^m
\inner{e_i, z_j}^2 = \sum_{j=t+1}^m \sum_{i=1}^m \inner{e_i, z_j}^2
= m-t,
\end{align*}
where the last equality is due to the fact that $e_i, 1\leq i \leq m$
forms a standard basis of $\real^m$, and $z_j$ is a unit vector.  So the claim
holds by the pigeonhole principle.

Now let $\Mat^T = \sum_{i=1}^{\lwk-1} \sigma_i u_i v_i^T$ with
$\sigma_1 \geq \cdots \geq \sigma_{\lwk-1}$ be the SVD of $\Mat^T$.
Take $W$ be the span of the first $t$ singular vectors
$v_1,\ldots,v_t$.  By taking dimension $m=\lwk-1$, the above claim implies there
exists some $e_\ell$ such that $\norm{e_\ell - \Pi_W(e_\ell)}_2^2 \geq
1 - \frac{t}{\lwk-1}$.  Let $v \defn
(h_{\ell,2},\ldots,h_{\ell,\ell-1}, 0,\ldots,0)$ so that the $\ell-1$
row of $\Mat^T$ can be written as $\Mat^T_{\cdot,\ell-1} = e_\ell^T +
v$.  Note by definition of matrix $\Mat$, we have $\norm{v}_2^2 \leq
a^2-9b^2$; thus
\begin{align*}
\norm{v - \Pi_W(v)}_2 \leq \norm{v}_2 \leq \sqrt{a^2-9b^2}/3b.
\end{align*}
Applying the triangle inequality yields
\begin{align}
\label{EqnDist}
\norm{\Mat_{\cdot,\ell-1} - \Pi_W(\Mat_{\cdot,\ell-1})}_2 \geq
\norm{e_\ell-\Pi_W(e_\ell)}_2 - \norm{v - \Pi_W(v)}_2 \geq 1 -
\frac{t}{\lwk - 1} - \frac{\sqrt{a^2 - 9b^2}}{3b}.
\end{align}

Define the matrix $\tilde{\Mat} \defn \sum_{i=1}^{t} \sigma_i u_i
v_i^T$, note that the rows of $\tilde{\Mat}$ are the projections of
the rows of $\Mat^T$ onto $W$, and therefore, the definition
of the operator norm implies that
\begin{align*}
    \norm{\Mat^T - \tilde{\Mat}}_{op} \geq \norm{\Mat_{\cdot,\ell-1} -
      \Pi_W(\Mat_{\cdot,\ell-1})}_2.
\end{align*}
It is easy to check that $\norm{\Mat^T - \tilde{\Mat}}_{op} =
\sigma_{t+1}$.  Combining with inequality~\eqref{EqnDist}, we have
\begin{align*}
\sigma_{t+1} \geq 1 - \frac{t}{\lwk - 1} - \frac{\sqrt{a^2 - 9b^2}}{3b},
\end{align*}
which completes the proof of inequality (ii).

\vspace*{0.5cm}
\noindent Combining our results from Parts I and II concludes the
proof of Lemma~\ref{LemPacking}.


\subsection{Proof of Lemma~\ref{LemHanaSleep}}
\label{AppLemHanaSleep}

We break our proof into two parts, corresponding to the two claims in
the lemma.

\paragraph{Proof of inequality~(i):}

This inequality is proved via a probabilistic argument.  Recall that
the ellipse norm is defined as $\enorm{x}^2 = \sum_{i=1}^{\usedim}
\frac{x_i^2}{\mu_i}$, and that we write $x \in \ellip$ to mean to
$\enorm{x} \leq 1$. We use $\NewMat$ to denote the diagonal matrix
with entries $(1/\mu_1, \ldots, 1/\mu_\usedim)$.

Let us define auxiliary vector 
\begin{align}
\label{EqnPerturbTilde}
 \tilde{\theta}^S \defn \thetadag +  \frac{b \epsilon}{2\sqrt{2s}} X V z^S,
\end{align}
where matrix $X \defn U\Mat$ is defined in Lemma~\ref{LemPacking}
and matrix $\Mat$ has eigen-decomposition $Q\Sigma V^T$.
We claim that the following two facts:
\begin{itemize}
\item[(a)] There exists a sign vector $z^S \in \{-1,0,1\}^{\lwk-1}$
  such that $\tilde{\theta}^S \in \ellip$.
\item[(b)] Defining $v_1 \defn \thetastar - \thetadag$, we then have
  $\thetastar + v_1\in \ellip$.
\end{itemize}
Taking these two facts as given for now, we can prove inequality~(i).
Noticing that $\thetadag + v_1 = \thetastar \in \ellip$ and $\thetadag
+ 2v_1 = \thetastar + v_1 \in \ellip$, by convexity of set $\ellip$,
if $\thetadag + u \in \ellip$ for some vector $u$ then
\begin{align*}
\at + \frac{1}{2} u = 
  \thetadag + v_1 + \frac{1}{2} u = \frac{1}{2}(\thetadag + u) + 
  \frac{1}{2}(\thetadag + 2v_1) \in \ellip.
\end{align*}
Letting $u \defn b \epsilon \frac{1}{\sqrt{s}} X V z^S$, then
$\theta^S \defn \thetastar + \frac{1}{2}b \epsilon \frac{1}{\sqrt{s}}
X V z^S = \at + \frac{1}{2} u$.  Then we have $\theta^S \in \ellip$.

\noindent It remains to establish claims (a) and (b) stated above.

\paragraph*{Proof of claim (a):}
Let $z^S \in \{-1, +1\}^{|S|}$ be a vector of i.i.d. Rademacher random
variables.  Expanding the square yields
\begin{align*}
\enorm{\tilde{\theta}^S}^2 & = (\thetadag)^T \NewMat \thetadag + 
\frac{b \epsilon}{\sqrt{2 s}}(\thetadag)^T \NewMat X V z^S +  \frac{b^2 \epsilon^2}{8 s}(X V z^S)^T \NewMat X V z^S\\
& = (\thetadag)^T \NewMat \thetadag + b^2 \epsilon^2 \frac{1}{8 s} (X V
z^S)^T \NewMat X V z^S,
\end{align*}
where the last inequality is due to the fact that $(\thetadag)^T \NewMat X =
0.$ Let us consider the expectation $\Exs \enorm{\tilde{\theta}^S}^2$.  Note
that
\begin{align*}
\Exs (X V z^S)^T \NewMat X V z^S &= \Exs ~\trace[(X V z^S)^T \NewMat X
  V z^S] \; = \; \Exs ~\trace[V^TX^T \NewMat X V z^S(z^S)^T],
\end{align*}
where the last inequality uses the property of the trace function.  By
linearity, we can switch the order of the expectation and trace
function, and doing so yields
\begin{align*}
  \Exs ~\trace[V^TX^T \NewMat X V z^S(z^S)^T] = \trace[
  \underbrace{V^TX^T \NewMat X V}_{\defn \AMat}
    \Exs [z^S(z^S)^T]] 
    = \sum_{i\in S} \AMat_{i}.
\end{align*} 
We claim $\sum_{i\in S} \AMat_i \leq 8s\max_{i} \{x_i^T\NewMat x_i\}$.
In order to show this, first note that 
\begin{align*}
  \trace[A] = \trace[V^TX^T \NewMat X V] = \trace[X^T \NewMat X]
  = \sum_{i=1}^{\lwk-1} x_i^T \NewMat x_i,
\end{align*}
using the fact that $V V^T= \Ind$.
Since $S$ is the subset of $F \defn \{m_1,m_1+1, \ldots,m_2\} \cap \Mdex(7(\lwk-1)/8)$, so for $t \defn \frac{7(\lwk-1)}{8}$, we have 
\begin{align*}
  ((\lwk-1) - t) A_{(t)} \leq \sum_{i=t}^{\lwk-1}A_{(i)}
  \leq \trace[A] \leq (\lwk-1) \max_{i} \{x_i^T\NewMat x_i\}.
\end{align*}
It implies that $A_{(t)} \leq 8 \max_{i} \{x_i^T\NewMat x_i\}$ which further verifies our claim. 

Putting the pieces together, we have
\begin{align*}
\Exs \enorm{\tilde{\theta}^S}^2 &= (\thetadag)^T \NewMat \thetadag +
 b^2 \epsilon^2 \frac{1}{8 s} \Exs(X Vz^S)^T \NewMat X V z^S
\\
& \stackrel{(i)}{\leq} (\thetadag)^T \NewMat \thetadag + b^2 \epsilon^2 \frac{1}{8 s} 8s\max_{i} \{x_i^T\NewMat x_i\}
\stackrel{(ii)}{=} \enorm{\thetadag + b\epsilon x_j}^2
\stackrel{(iii)}{\leq} 1,
\end{align*}
where step (i) uses claim on the $A_i$s, step (ii) uses fact that 
$\thetadag \NewMat x_j = 0$ and 
step (iii) uses the fact that $\thetadag + b\epsilon x_j \in \ellip$.  Since the average squared ellipse norm is at most one,
there must exist at least one $\tilde{\theta}^S$ with
$\enorm{\tilde{\theta}^S}^2 \leq 1$, as claimed.

\paragraph*{Proof of claim (b):}
Since $v_1 = \thetastar - \thetadag$, the results of
Lemma~\ref{LemPacking} guarantee that
\begin{align*}
\enorm{v_1}^2 = \sum_{i=1}^\usedim
\left(\frac{r(\epsilon)}{r(\epsilon)+\mu_i}\right)^2(\at_i)^2 & =
\sum_{i=1}^\usedim \frac{r(\epsilon)^2 \mu_i}{(r(\epsilon)+\mu_i)^2}
\frac{(\at_i)^2}{\mu_i} \\
& \leq \enorm{\at}^2 \; \max_{i = 1, \ldots, \usedim}
\left\{\frac{r(\epsilon)^2 \mu_i}{(r(\epsilon)+\mu_i)^2}\right\}.
\end{align*}
Since $(r(\epsilon)+\mu_i)^2 \geq 4r(\epsilon)\mu_i$, we have
$\enorm{v_1} \leq \frac{\sqrt{r(\epsilon)}}{2}\enorm{\at}$.  Triangle
inequality guarantees that
\begin{align*}
  \enorm{\at + v_1} \leq \enorm{\at} + \enorm{v_1} \leq 
  \left(1+\frac{\sqrt{r(\epsilon)}}{2}\right)\enorm{\at} \leq 1,
\end{align*}
where the last inequality uses assumption that 
$r(\epsilon) = \Rfun(\epsilon) < (\enorm{\at}^{-1} - 1)^2$
and $\enorm{\at} \leq 1.$


\paragraph{Proof of inequality~(ii):}
Letting $\DeltaTil^S \defn \tilde{\theta}^S - \thetadag$ and 
by construction $\Delta^S = \frac{1}{2} \DeltaTil^S.$
By definition, $\ltwo{\DeltaTil}^2 = \frac{1}{8s} b^2 \epsilon^2
\norm{U\Sigma Vz^S}_2^2.$ Substituting matrix $\Mat$ as its
decomposition $Q\Sigma V^T$ yields
\begin{align} \label{EqnHalloween}
  \norm{\tilde{\theta}^S - \thetadag}_2^2 & = \frac{1}{8s} b^2 \epsilon^2 
  \norm{UQ\Sigma z^S }_2^2 = 
 \frac{1}{8s} b^2 \epsilon^2 \norm{\Sigma z^S}_2^2,
\end{align}
where the last inequality uses the orthogonality of matrices $U, Q$.

Recall that set $S$ is a subset of $\{m_1,\ldots,m_2\}$.  Using
eigenvalue bound~\eqref{EqnEignBound}(ii) from part (e) of
Lemma~\ref{LemPacking}, we have $\norm{\Sigma z^S}_2^2 \geq (1 -
\frac{m_2}{\tmpdim} - \frac{\sqrt{a^2 - 9b^2}}{3b})^2 \ltwo{z^S}^2$.  Combining with equality~\eqref{EqnHalloween} yields
inequality~(ii).\\

\noindent Finally, putting the two parts together completes the proof
of Lemma~\ref{LemHanaSleep}.


\subsection{Proof of Lemma~\ref{LemPackingLB}}
\label{AppLemPackingLB}

Let us denote the cardinality of $\mathcal{S}$ as $N \defn
\binom{m_2-m_1+1}{s}.$ Then the expectation can be written as
\begin{align*}
\Exs \exp\big(\frac{\inprod{\eta}{\eta'}}{\sigma^2}\big) = \frac{1}{N^2}
\sum_{S, S' \in \KposSet} \exp\big(\frac{1}{\sigma^2}
\inprod{\Delta^S}{\Delta^{S'}}\big).
\end{align*}
From the definition~\eqref{EqnPerturb} of $\theta^S$, we have
\begin{align*}
\inprod{\Delta^S}{\Delta^{S'}} = \frac{1}{32}b^2 \epsilon^2
\frac{1}{s}(z^S)^T V^T \Mat^T U^T U \Mat V z^{S'} = \frac{1}{32}b^2
\epsilon^2 \frac{1}{s} (z^S)^T \Sigma^2 z^{S'},
\end{align*}
where we used the relation $\Mat^T \Mat= V \Sigma^2 V^T$, and the
orthogonality of the matrices $V$ and $U$.  In order to further
control the right hand side, note that
\begin{align*}
(z^S)^T \Sigma^2 z^{S'} = \sum_{i\in S\cap S'} \nu_i^2 z^S_iz^{S'}_i
  \leq \nu^2_{m_1} \sum_{i\in S\cap S'} |z^S_iz^{S'}_i| = \nu^2_{m_1}
  |S \cap S'|.
\end{align*}
Combining these three inequalities above yields
\begin{align*}
\Exs \exp \big(\frac{\inprod{\eta}{\eta'}}{\sigma^2}\big) 
\leq \frac{1}{N^2}
\sum_{S,S' \in \KposSet} \exp \big(\frac{1}{32\sigma^2} b^2 \epsilon^2
\nu^2_{m_1} \frac{|S \cap S'|}{s}\big).
\end{align*}

Let us now further compute the right hand side above.  Recall that we
denote $\tmpdim \defn m_2-m_1+1$.  Note that intersection cardinality
$|S \cap S'|$ takes values in $\{0, 1, \ldots, s \}$.  Given every set
$S$ and integer $i \in \{0, 1, \ldots, s \}$, the number of $S'$ such
that $|S \cap S'| = i$ equals to ${s \choose i}{\tmpdim-s \choose
  s-i}$.  
Consequently, if we let $\lambda \defn \frac{1}{32\sigma^2} b^2 \epsilon^2
\nu^2_{m_1}$, we obtain
\begin{align}
\label{EqnChiSquareUpper}
\Exs \exp\big(\frac{\inprod{\eta}{\eta'}}{\sigma^2}\big) 
= {\tmpdim \choose
  s}^{-1} \sum_{i=0}^s {s \choose i} {\tmpdim-s \choose s-i}
e^{\lambda i/s} = \sum_{i=0}^s \frac{\SPEC_i z^i}{i !},
\end{align}
where $z \defn e^{\lambda/s}$ and $\SPEC_i \defn
\frac{(s!(\tmpdim-s)!)^2}{((s-i)!)^2 \tmpdim!(\tmpdim-2s+i)!}$.

Let us set integer $s \defn \lfloor \sqrt{\tmpdim} \rfloor$.  We claim
quantity $\SPEC_i$ satisfies the following bound
\begin{align} \label{EqnNonNe-Ai}
\SPEC_i \leq \exp\Big(- (1-\frac{1}{\sqrt{\tmpdim}})^2 +
\frac{2i}{\sqrt{\tmpdim}-1}\Big) \qquad \mbox{for all $i \in \{0, 1,
  \ldots, s\}$.}
\end{align} 
Taking expression~\eqref{EqnNonNe-Ai} as given for now and plugging
into inequality \eqref{EqnChiSquareUpper}, we have
\begin{align*}
\Exs \exp(\frac{\inprod{\eta}{\eta'}}{\sigma^2}) \leq \exp \Big(-
(1-\frac{1}{\sqrt{\tmpdim}})^2 \Big)\sum_{i=0}^{s}
\frac{(z\exp(\frac{2}{\sqrt{\tmpdim}-1}))^i}{i !}
& \, \stackrel{(a)}{\leq} \, \exp \Big(-
(1-\frac{1}{\sqrt{\tmpdim}})^2 \Big) \exp \left( z
\exp(\frac{2}{\sqrt{\tmpdim}-1})\right)\\
& \, \stackrel{(b)}{\leq} \, \exp \left(-
\left(1-\frac{1}{\sqrt{\tmpdim}}\right)^2 + \exp
\left(\frac{2+\lambda}{\sqrt{\tmpdim}-1}\right) \right),
\end{align*}
where step (a) follows from the standard power series expansion $e^x =
\sum_{i=0}^\infty \frac{x^i}{i!}$ and step (b) follows by $z =
e^{\lambda/s}$ and $s = \lfloor \sqrt{\tmpdim} \rfloor >
\sqrt{\tmpdim} - 1$.  We have thus established
inequality~\eqref{EqnPackingLB}.

It only remains to check inequality~\eqref{EqnNonNe-Ai} for $\SPEC_i$.
Using the fact that $1 - x \leq e^{-x}$, we have
\begin{align}
\label{EqnHfun1}  
\SPEC_0 = \frac{((\tmpdim - s)!)^2}{\tmpdim! (\tmpdim - 2s)!}  &= (1 -
\frac{s}{\tmpdim}) (1 - \frac{s}{\tmpdim-1})\cdots(1 -
\frac{s}{\tmpdim-s+1}) \leq \exp( - s\sum_{i=1}^s
\frac{1}{\tmpdim-s+i}).
\end{align}
Recalling that $s \defn \lfloor \sqrt{\tmpdim} \rfloor$, then we can
bound the sum in expression~\eqref{EqnHfun1} as
\begin{align*}
s\sum_{i=1}^s \frac{1}{\tmpdim-s+i} \geq s \sum_{i=1}^{s}
\frac{1}{\tmpdim} = \frac{s^2}{\tmpdim} \geq
(1-\frac{1}{\sqrt{\tmpdim}})^2,
\end{align*}
which, when combined with inequality~\eqref{EqnHfun1}, implies that
$\SPEC_0 \leq \exp( - (1-\frac{1}{\sqrt{\tmpdim}})^2 )$.

Moreover, direct calculations yield that we have
\begin{align} 
\label{EqnHfun2} 
\frac{\SPEC_i}{\SPEC_{i-1}} = \frac{(s-i+1)^2}{\tmpdim-2s+i} \; \leq
\; \frac{\SPEC_1}{\SPEC_0} \quad \mbox{for all $i = 1, \ldots, s$,}
\end{align}
  where the last inequality follows from the fact that
  $\frac{(s-i+1)^2}{\tmpdim-2s+i}$ is non-increasing with index $i$.
  Therefore, recalling that $s = \lfloor \sqrt{\tmpdim} \rfloor$, we
  have
\begin{align*}
\frac{\SPEC_i}{\SPEC_{i-1}} \leq \frac{G_1}{G_0} \; \leq \;
\frac{\tmpdim}{\tmpdim-2\sqrt{\tmpdim}+1} = (1 +
\frac{1}{\sqrt{\tmpdim}-1})^2 \leq \exp(\frac{2}{\sqrt{\tmpdim}-1}),
\end{align*}
where the last inequality follows from $1+x \leq e^x.$ Putting pieces
together validates bound \eqref{EqnNonNe-Ai} thus finishing the proof
of inequality~\eqref{EqnPackingLB}.


\section{Proof of Lemma~\ref{LemCalculation}}
\label{AppCalc}

Letting $\Delta \defn \sqrt{ \frac{\delta^2 - \mu_{\mind+1}}{1-t} +
  \frac{t(\at_s)^2}{(1-t)^2}} - \frac{t\at_s}{1-t}$, we then have
$\theta_{\mind+1}^2 = \delta^2 - \Delta^2$.  Direct calculations yield
\begin{align*}
  \Delta = \frac{\frac{\delta^2 - \mu_{\mind+1}}{1-t} +
    \frac{t(\at_s)^2}{(1-t)^2} - \frac{t^2(\at_s)^2}{(1-t)^2} }
         { \sqrt{ \frac{\delta^2 -
               \mu_{\mind+1}}{1-t} + \frac{t(\at_s)^2}{(1-t)^2}} + \frac{t\at_s}{1-t}}
  = \frac{\frac{\delta^2 - \mu_{\mind+1}}{1-t} + \frac{t(\at_s)^2}{1-t}}
    {\sqrt{ \frac{\delta^2
     - \mu_{\mind+1}}{1-t} + \frac{t(\at_s)^2}{(1-t)^2}}
     + \frac{t\at_s}{1-t}}.
\end{align*}
Re-organizing the terms yields 
\begin{align*}  
  \Delta &= \frac{\delta^2 - \mu_{\mind+1} + t(\at_s)^2}{ t\at_s +
    \sqrt{(\delta^2 - \mu_{\mind+1})(1-t) + t(\at_s)^2} }
  = \delta\Big(\frac{1 + \frac{t(\at_s)^2-
      \mu_{\mind+1}}{\delta^2}}{\frac{t\at_s}{\delta} + \sqrt{1 - t + 
      \frac{t\mu_{\mind+1}}{\delta^2} + \frac{t(\at_s)^2-
        \mu_{\mind+1}}{\delta^2} }} \Big).
\end{align*}
Now let us first analyze the denominator.  Since $t$ is defined as
$\frac{\mu_{\mind+1}}{\mu_s}$, we obtain
$\frac{t\mu_{\mind+1}}{\delta^2} =
\frac{\mu^2_{\mind+1}}{\delta^2\mu_s}$; by $\at_s \leq \sqrt{\mu_s}$,
we obtain that $\frac{t\at_s }{\delta} \leq
\frac{t\sqrt{\mu_s}}{\delta} \leq \frac{\mu_{\mind+1}}{\delta
  \sqrt{\mu_s}}$.  Plugging into the expression of $\Delta$ gives
\begin{align*}
  \frac{\Delta}{\delta} \geq \frac{1 + \frac{t(\at_s)^2-
      \mu_{\mind+1}}{\delta^2}}{ \frac{\mu_{\mind+1}}{\delta
      \sqrt{\mu_s}} + \sqrt{1+ \frac{\mu^2_{\mind+1}}{\delta^2\mu_s} +
      \frac{t(\at_s)^2- \mu_{\mind+1}}{\delta^2}-t }}.
\end{align*}
Moreover, since $\mind$ is chosen as the maximum index that satisfies
$\mu_k^2 \geq \frac{1}{64}\delta^2\mu_s$, it is guaranteed that
$\mu_{\mind+1}^2 \leq \frac{1}{64}\delta^2\mu_s$ which further implies
that
\begin{align}
\label{EqnBrunch}
  \frac{\Delta}{\delta} \geq \frac{1 + \frac{t(\at_s)^2-
      \mu_{\mind+1}}{\delta^2}}{ \frac{1}{8} + \sqrt{1+ \frac{1}{64} +
      \frac{t(\at_s)^2- \mu_{\mind+1}}{\delta^2}-t }}.
\end{align}

Now, in order to control $\frac{\Delta}{\delta}$, we only need to control
quantity $\frac{t(\at_s)^2- \mu_{\mind+1}}{\delta^2}$.  Recall that
for any \mbox{$\delta > \newepscritu(s, \ellip)$}, we only consider those
$\at$ satisfying $\at_s \geq \sqrt{\mu_s} - \newepscritu(s, \ellip)$.
Let us write $\at_s = \sqrt{\mu_s} - c\delta$ with $c\in [0,1]$, then
\begin{align*}
  \frac{t(\at_s)^2- \mu_{\mind+1}}{\delta^2} = t\frac{(\sqrt{\mu_s} -
    c\delta)^2- \mu_s}{\delta^2} = 
    \underbrace{-2c
  \frac{\mu_{\mind+1}}{\delta\sqrt{\mu_s}}}_{\defn x} + c^2 t.
\end{align*}
Here $x \defn -2c \frac{\mu_{\mind+1}}{\delta\sqrt{\mu_s}} \geq
-2\frac{\mu_{\mind+1}}{\delta\sqrt{\mu_s}} \geq - \frac{1}{4} $ since
$c \leq 1$ and $\mu_{\mind+1}^2 \leq \frac{1}{64}\delta^2\mu_s$.  Now
we are ready to substitute this expression into
inequality~\eqref{EqnBrunch}.  Therefore it is guaranteed that
\begin{align*}
  \frac{\Delta}{\delta} \geq \frac{1 + x + c^2t}{\frac{1}{8} +
    \sqrt{\frac{65}{64} + x + c^2t -t}} \geq \frac{1 + x}{\frac{1}{8}
    + \sqrt{\frac{65}{64} + x}} \geq \min_{x\geq -1/4} \Big\{\frac{1 +
    x}{\frac{1}{8} + \sqrt{\frac{65}{64} + x}} \Big\} = \frac{3}{4}.
\end{align*}
Consequently, we have
\begin{align*}
\theta^2_{\mind+1} = \delta^2 - \Delta^2 \leq \delta^2 -
(\frac{3\delta}{4})^2 \leq \frac{1}{2}\delta^2.
\end{align*}


\section{Auxiliary results on the critical radii}
\label{AppAux}
In this appendix, we collect a number of auxiliary results concerning
the upper and low critical inequalities, as defined in
equations~\eqref{EqnRadCritU} and~\eqref{EqnRadCritL}, respectively.


\subsection{Existence and uniqueness}
\label{AppSolution}

First, let us show that there exists a unique and strictly positive
solution to inequalities~\eqref{EqnRadCritU} and~\eqref{EqnRadCritL}.

Let us prove this claim for the solution of
inequality~\eqref{EqnRadCritU}, since the argument for
inequality~\eqref{EqnRadCritL} is entirely analogous.  It suffices to
show that the function $g(t) \defn \upk(t, \thetastar, \ellip)$ is
non-increasing on the positive real line.  Since the function $t
\mapsto 1/t$ is strictly decreasing on the positive real line, we are
then guaranteed that the function $t \mapsto \sqrt{g(t)}/t =
\sqrt{\upk(t, \thetastar, \ellip)}/t$ is strictly decreasing on the
positive real line.  This property ensures that
inequality~\eqref{EqnRadCritU} has a unique and strictly positive
solution: the left-side function $t$ is strictly increasing, whereas
the right-side function is strictly decreasing.

It remains to show that $g(t) = \upk(t, \thetastar, \ellip)$ is
non-increasing.  Equivalently, we need to show that for any $t > 0$,
we have
\begin{align*}
  g(c t) \; = \; \upk(c t, \thetastar, \ellip) \geq \upk(t,
  \thetastar, \ellip) \; = \; g(t) \qquad \mbox{for all $c\in (0,
    1]$.}
\end{align*}
By linearity of the Kolmogorov width, we have
\begin{align}
\label{EqnHanaPrincess}
\upk(t, \thetastar, \ellip) \defn \arg \min_{1 \leq k \leq \usedim}
\braces*{\kwidth_k(\ellip_{\thetastar} \cap \Ball(t)) \leq \frac{1}{2}
  t} = \arg \min_{1 \leq k \leq \usedim}
\braces*{\kwidth_k(c(\ellip_{\thetastar} \cap \Ball(t))) \leq
  \frac{1}{2} c \, t}.
\end{align}
Given any vector $v\in \ellip_{\thetastar} \cap \Ball(t)$, convexity
ensures that $c v \in \ellip_{\thetastar} \cap \Ball(c t)$, which is
equivalent to the containment $c \, \big(\ellip_{\thetastar} \cap
\Ball(t) \big) \subset \ellip_{\thetastar} \cap \Ball(c t)$.  This
containment implies that
\begin{align*}
  \kwidth_k(c(\ellip_{\thetastar} \cap \Ball(t))) \leq
  \kwidth_k(\ellip_{\thetastar} \cap \Ball(c t)).
\end{align*}
Combined with our earlier inequality~\eqref{EqnHanaPrincess}, we
conclude that $\upk(t, \thetastar, \ellip) \leq \upk(c t, \thetastar,
\ellip)$, as desired.


\subsection{Well-definedness of the function $\Rfun$}
\label{AppRfun}

In this appendix, we verify that the function $\Rfun$ from
equation~\eqref{EqnRfun} is well-defined.  In order to provide
intuition, \autoref{fig:Rfun} provides an illustration of $\Rfun$.
\begin{figure}[h]
  \centering \includegraphics[width=0.45\textwidth]{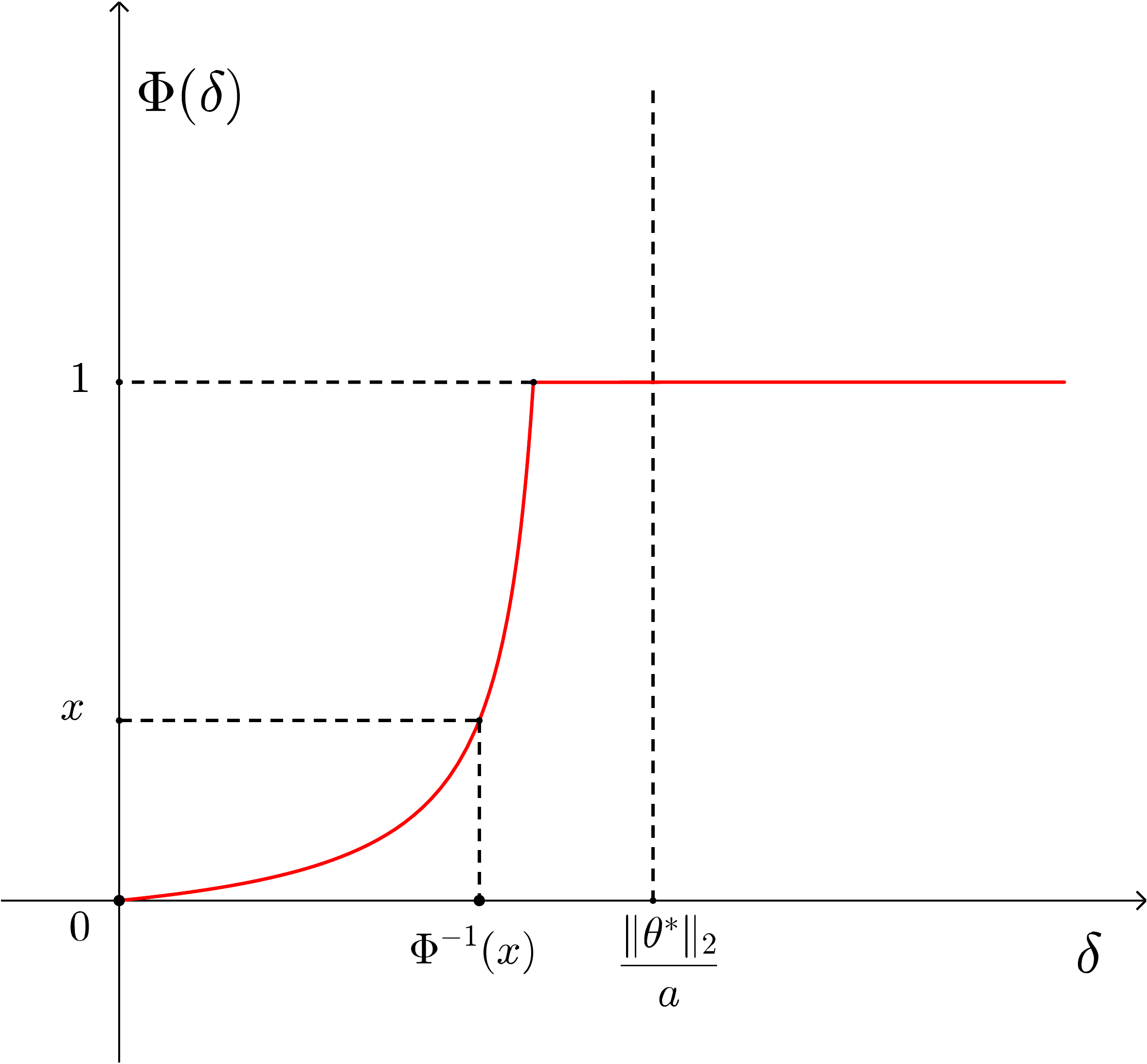}
  \caption{Illustration of the function $\Rfun$.}
  \label{fig:Rfun}
\end{figure}

We begin with the case when $\delta < \ltwo{\at}/a$.  For notation
simplicity, let $r(\delta) \defn \min \Big \{ r \geq 0 \, \mid a^2
\delta^2 \leq \sum_{i=1}^\usedim \frac{r^2}{(r+\mu_i)^2} (\at_i)^2
\Big \}.$ Since each $\mu_i \geq 0$, function \mbox{$f: r \to
  \sum_{i=1}^\usedim \frac{r^2}{(r+\mu_i)^2}(\at_i)^2$} is
non-decreasing in $r$.  It is easy to check that
  \begin{align*}
    \lim_{r\to 0^+} f(r) = 0,\qquad \text{ and }\lim_{r\to \infty} f(r) = \ltwo{\at}^2.
  \end{align*}
  Then quantity $r(\delta)$ is uniquely defined and positive whenever
  $\delta < \ltwo{\at}/a$.  Note that as $\delta \to
  \frac{\ltwo{\at}}{a}$, $a^2 \delta^2 \to \ltwo{\at}^2$ therefore
  $r(\delta) \to \infty.$
  
  It is worth noticing that given any $\at$ where $\ltwo{\at}$ does
  not depend on $\delta$, $r$ goes to zero when $\delta \to 0$, namely
  $\lim_{\delta \to 0^+} \Rfun(\delta) = 0.$



\end{document}

%% file: AMS_commands.tex
\theoremstyle{plain}

\newtheorem{theo}{Theorem}[section]

\newtheorem{lem}{Lemma}[section]
\newtheorem{prop}{Proposition}[section]
\newtheorem{cor}{Corollary}[section]

\theoremstyle{definition} 

\newtheorem{nota}{Notation}[section]
\newtheorem{de}{Definition}[section]
\newtheorem{exa}{Example}[section]
\newtheorem{as}{Assumption}[section]
\newtheorem{alg}{Algorithm}[section]

\newcommand{\btheo}{\begin{theo}}
\newcommand{\bde}{\begin{de}}
\newcommand{\ble}{\begin{lem}}
\newcommand{\bpr}{\begin{prop}}
\newcommand{\bno}{\begin{nota}}
\newcommand{\bex}{\begin{exa}}
\newcommand{\bcor}{\begin{cor}}
\newcommand{\spro}{\begin{proof}}
\newcommand{\bas}{\begin{as}}
\newcommand{\balg}{\begin{alg}}

\newcommand{\etheo}{\end{theo}}
\newcommand{\ede}{\end{de}}
\newcommand{\ele}{\end{lem}}
\newcommand{\epr}{\end{prop}}
\newcommand{\eno}{\end{nota}}
\newcommand{\eex}{\end{exa}}
\newcommand{\ecor}{\end{cor}}
\newcommand{\fpro}{\end{proof}}
\newcommand{\eas}{\end{as}}
\newcommand{\ealg}{\end{alg}}

\theoremstyle{plain}

\newtheorem{theos}{Theorem}
\newtheorem{props}{Proposition}
\newtheorem{lems}{Lemma}
\newtheorem{cors}{Corollary}

\theoremstyle{definition}
\newtheorem{exas}{Example}
\newtheorem{algs}{Algorithm}
\newtheorem{asss}{Assumption}
\newtheorem{defns}{Definition}

\newcommand{\btheos}{\begin{theos}}
\newcommand{\etheos}{\end{theos}}
\newcommand{\bprops}{\begin{props}}
\newcommand{\eprops}{\end{props}}
\newcommand{\bdes}{\begin{defns}}
\newcommand{\edes}{\end{defns}}
\newcommand{\blems}{\begin{lems}}
\newcommand{\elems}{\end{lems}}
\newcommand{\bcors}{\begin{cors}}
\newcommand{\ecors}{\end{cors}}
\newcommand{\bexs}{\begin{exas}}
\newcommand{\eexs}{\end{exas}}
\newcommand{\balgs}{\begin{algs}}
\newcommand{\ealgs}{\end{algs}}
\newcommand{\bass}{\begin{asss}}
\newcommand{\eass}{\end{asss}}

%% file: STAT_macros.tex
\newcommand{\numobs}{\ensuremath{n}}
\newcommand{\usedim}{\ensuremath{d}}

\newcommand{\thetastar}{\ensuremath{\theta^*}}

\newcommand{\real}{\ensuremath{\mathbb{R}}}
\newcommand{\defn}{\ensuremath{: \, =}}

\newcommand{\inprod}[2]{\ensuremath{\langle #1 , \, #2 \rangle}}

\newcommand{\Exs}{\ensuremath{\mathbb{E}}}

\newcommand{\NORMAL}{\ensuremath{N}}

\long\def\comment#1{}

\newcommand{\ltwo}[1]{\ensuremath{\|#1\|_2}}

\newcommand{\HACKPROOF}{\begin{proof}}
\newcommand{\HACKENDPROOF}{\end{proof}}

\newcommand{\Ball}{\ensuremath{\mathbb{B}}}

\newcommand{\qprob}{\ensuremath{\mathbb{Q}}}

 \newcommand{\rad}{\ensuremath{r}}

\newcommand{\Ind}{\ensuremath{\mathbb{I}}}

\newcommand{\trace}{\ensuremath{\mbox{trace}}}

\newcommand{\widgraph}[2]{\includegraphics[keepaspectratio,width=#1]{#2}}

\newlength{\widebarargwidth}
\newlength{\widebarargheight}
\newlength{\widebarargdepth}
